\documentclass{article}

\usepackage{authblk}
\usepackage{a4wide}
\usepackage{amsthm}
\usepackage{graphicx}
\usepackage{amssymb}
\usepackage{amsmath}
\usepackage{ascmac}
\usepackage{setspace}
\usepackage{float}
\usepackage[dvips,usenames]{color}
\usepackage{colortbl}
\usepackage{algorithm}
\usepackage{algorithmic}
\usepackage{setspace}

\newtheorem{Theorem}[equation]{Theorem}

\newtheorem{Lemma}[equation]{Lemma}

\theoremstyle{definition}
\newtheorem{Definition}[equation]{Definition}
\newtheorem{Conjecture}[equation]{Conjecture}
\theoremstyle{remark}
\newtheorem{Remark}[equation]{Remark}
\numberwithin{equation}{section}

\DeclareMathOperator{\ev}{ev}
\DeclareMathOperator{\id}{id}

\DeclareMathOperator{\ad}{ad}

\DeclareMathOperator{\row}{row}
\DeclareMathOperator{\col}{col}

\newcommand{\ve}{\varepsilon}

\newcommand{\tss}{\hspace{1pt}}

\allowdisplaybreaks

\begin{document}
\title{Affine Yangians and some cosets of non-rectangular $W$-algebras}
\author{Mamoru Ueda\thanks{mueda@ualberta.ca}}
\affil{Department of Mathematical and Statistical Sciences, University of Alberta, 11324 89 Ave NW, Edmonton, AB T6G 2J5, Canada}
\date{}
\maketitle
\begin{abstract}
We construct a homomorphism from the affine Yangian associated with $\widehat{\mathfrak{sl}}(q_u-q_{u+1})$ to the universal enveloping algebra of a $W$-algebra associated with $\mathfrak{gl}(\sum_{s=1}^l\limits q_s)$ and a nilpotent element of type $(1^{q_1-q_2},2^{q_2-q_3},\dots,(l-1)^{q_{l-1}-q_l},l^{q_l})$
by using the coproduct and evaluation map for the affine Yangian and the Miura map for $W$-algebras. As an application, we give a homomorphism from the affine Yangian to some cosets of non-rectangular $W$-algebras.
\end{abstract}
\section{Introduction}
The AGT conjecture suggests the existence of a representation of the principal $W$-algebra of type $A$ on the equivariant homology space of the moduli space of $U(r)$-instantons. Schiffmann and Vasserot \cite{SV} gave this representation by using an action of the Yangian associated with $\widehat{\mathfrak{gl}}(1)$ on this equivariant homology space. 
More generally, Crutzig-Diaconescu-Ma \cite{CE} conjectured that an action of an iterated $W$-algebra of type $A$ on the equivariant homology space of the affine Laumon space will be given through an action of an shifted affine Yangian constructed in \cite {FT}.

In order to resolve the Crutzig-Diaconescu-Ma's conjecture, we need to construct a surjective homomorphism from the shifted affine Yangian to the universal enveloping algebra of the iterated $W$-algebra of type $A$. The iterated $W$-algebra of type $A$ is expected to be the tensor of a $W$-algebra of type $A$ and $\beta\gamma$-systems. 
Thus, the Crutzig-Diaconescu-Ma's conjecture can be considered as the affine analogue of Brundan-Kleshchev \cite{BK}, which gives a surjective homomorphism from the shifted Yangian to a finite $W$-algebra of type $A$. However, the shifted affine Yangian is so complicated that it is difficult to construct the homomorphism of the Crutzig-Diaconescu-Ma's conjecture directly. In finite setting, there exist homomorphisms from finite Yangians to finite $W$-algebras of type $A$, which are restirictions of the one in \cite{BK}. One of these homomorphisms was given by De Sole-Kac-Valeri \cite{DKV} by using the Lax operator. In \cite{U7} and \cite{U11}, the author constructed the affine analogue of the homomorphism in De Sole-Kac-Valeri \cite{DKV}. Let us take a positive integer $N$ and its partition:
\begin{gather*}
N=q_1+q_2+\cdots+q_l,\qquad q_1\geq q_{2}\geq \cdots\geq q_l
\end{gather*}
and set $f$ as a nilpotent element of type $(1^{q_1-q_2},2^{q_2-q_3},\dots,(l-1)^{q_{l-1}-q_l},l^{q_l})$.
In \cite{U11}, we constructed a homomorphism from the affine Yangian associated with $\widehat{\mathfrak{sl}}(q_l)$ to the universal enveloping algebra of a $W$-algebra $\mathcal{W}^k(\mathfrak{gl}(N),f)$, which is a $W$-algebra associated with $\mathfrak{gl}(N)$ and $f$ by using the the coproduct and evaluation map for the affine Yangian and the Miura map for non-rectangular $W$-algebras. Moreover, in \cite{U11}, we show that the coproduct for the affine Yangian is compatible with the parabolic induction for a non-rectangular $W$-algebra via this homomorphism.

In \cite{U9}, we gave the generalization of \cite{U7}.
We constructed a homomorphism from the affine Yangian associated with $\widehat{\mathfrak{sl}}(q_u-q_{u+1})$ to the universal enveloping algebra of $\mathcal{W}^k(\mathfrak{gl}(N),f)$. In this article, we gave another proof to \cite{U9}. In \cite{U8} and \cite{U10}, we gave a homomorphism
\begin{align*}
\Psi_1\otimes\Psi_2\colon Y_{\hbar,\ve}(\widehat{\mathfrak{sl}}(n))\otimes Y_{\hbar,\ve+n\hbar}(\widehat{\mathfrak{sl}}(m))\to Y_{\hbar,\ve}(\widehat{\mathfrak{sl}}(m+n)).
\end{align*}
Using $\Psi_1\otimes\Psi_2$, we obtain a homomorphism from the tensor product $\bigotimes_{1\leq i\leq l}Y_{\hbar,\ve+\sum_{u=1}^{i-1}n_u\hbar}(\widehat{\mathfrak{sl}}(n_i))$ to $Y_{\hbar,\ve}(\widehat{\mathfrak{sl}}(\sum_in_i))$.
Its image is the affine version of the Levi sualbgera of the finite Yangian of type $A$ given in \cite{BK0}.
Let us take positive integers $\{n_a\}_{a=1}^{s+1}$ and $\{b_a\}_{a=1}^s$by
\begin{equation*}
n_a>n_{a+1},\ n_{a-1}=q_{b_{a-1}}>q_{b_{a-1}+1}=\cdots=q_{b_{a}}=n_{a}>q_{b_{a}+1}=n_{a+1},\ n_{s+1}=0.
\end{equation*}
By using the coproduct for the affine Yangian and $\Psi_1\otimes\Psi_2$, we can construct a homomorphism
\begin{equation*}
\Delta_l\colon\bigotimes_{1\leq a\leq s}Y_{\hbar,\ve-(n_a-q_l)\hbar}(\widehat{\mathfrak{sl}}(n_a-n_{a+1}))\to\widehat{\bigotimes}_{1\leq v\leq l}Y_{\hbar,\ve-(q_v-q_l)\hbar}(\widehat{\mathfrak{sl}}(q_v)),
\end{equation*}
where $\widehat{\bigotimes}_{1\leq v\leq l} Y_{\hbar,\ve-(q_v-q_l)\hbar}(\widehat{\mathfrak{sl}}(q_v))$ is the standard degreewise completion of the tensor of affine Yangians.
The Miura map (\cite{KW1}) induces a homomorphism
\begin{equation*}
\widetilde{\mu}\colon \mathcal{U}(\mathcal{W}^k(\mathfrak{gl}(N),f))\hookrightarrow\widehat{\bigotimes}_{1\leq i\leq l}U(\widehat{\mathfrak{gl}}(q_i)),
\end{equation*}
where $\widehat{\bigotimes}_{1\leq i\leq l}U(\widehat{\mathfrak{gl}}(q_i))$ is the standard degreewise completion of $\bigotimes_{1\leq i\leq l}U(\widehat{\mathfrak{gl}}(q_i))$. We can connect $\Delta_l$ to the Miura map $\widetilde{\mu}$ by using the evalution map for the Yangian:
\begin{equation*}
\ev^n\colon Y_{\hbar,\ve-(n_a-q_l)\hbar}(\widehat{\mathfrak{sl}}(n))\to \mathcal{U}(\widehat{\mathfrak{gl}}(n)),
\end{equation*}
where $\mathcal{U}(\widehat{\mathfrak{gl}}(n))$ is the standard degreewise completion of the universal enveloping algebra of $\widehat{\mathfrak{gl}}(n)$. 
\begin{Theorem}\label{A}
There exists a homomorphism
\begin{equation*}
\Phi\colon \bigotimes_{1\leq a\leq s}Y_{\hbar,\ve-(n_a-q_l)\hbar}(\widehat{\mathfrak{sl}}(n_a-n_{a+1}))\to\mathcal{U}(\mathcal{W}^k(\mathfrak{gl}(N),f))
\end{equation*}
satisfying the relation
\begin{equation*}
\widetilde{\mu}\circ\Phi=\bigotimes_{1\leq i\leq l}\ev^{q_i}\circ\Delta_l.
\end{equation*}
\end{Theorem}
The homomorphism $\Phi$ is corresponding to the one from the tensor product of finite Yangians of type $A$, the Levi subalgebra of the finite Yangian, to a finite $W$-algebra of type $A$, which is the restriction of the homomorphism in \cite{BK}. Thus, it is expected that we will obtain a homomorphism of Crutzig-Diaconescu-Ma's conjecture by extending $\Phi$ to the shifted affine Yangian. Also, there exists another version of this conjecture which notes the existence of a surjective homomorphism from the shifted affine Yangian to the universal enveloping algebra of a $W$-algebra of type $A$ if we change the definition of the shifted affine Yangian properly. Since $\bigotimes_{1\leq a\leq s}Y_{\hbar,\ve-(n_a-q_l)\hbar}(\widehat{\mathfrak{sl}}(n_a-n_{a+1}))$ is corresponding to the Levi subalgebra of the finite Yangian, we expect that Theorem~\ref{A} will lead the new presentation of the affine Yangian corresponding to the parabolic presentation and be helpful for the resolution of the another version of the Crutzig-Diaconescu-Ma's conjecture.

The homomorphism $\Phi$ induces a homomorphism
\begin{equation*}
\Phi_a\colon Y_{\hbar,\ve-(n_a-q_l)\hbar}(\widehat{\mathfrak{sl}}(n_a-n_{a+1}))\to\mathcal{U}(\mathcal{W}^k(\mathfrak{gl}(N),f))
\end{equation*}
for $1\leq a\leq s$, which are homomorphisms given in \cite{U7} and \cite{U9}. 
By the construction way of $\Delta_l$, we obtain the following theorem:
\begin{Theorem}\label{B}
The images of $\Phi_a$ are commutative with each other. 
\end{Theorem}
Theorem~\ref{B} will be connected to the Gaiotto-Rapcak's triality.
Gaiotto and Rapcak \cite{GR} introduced a vertex algebra called the $Y$-algebra. The $Y$-algebra can be interpreted as a truncation of $\mathcal{W}_{1+\infty}$-algebra (\cite{GG}) whose universal enveloping algebra is isomorphic to the affine Yangian of $\widehat{\mathfrak{gl}}(1)$ up to suitable completions (see \cite{AS}, \cite{T} and \cite{MO}). Gaiotto-Rapcak \cite{GR} conjectured a triality of the isomorphism of $Y$-algebras. Let us take a nilpotent element $f\in\mathfrak{sl}(m+n)$ of type $(n^1,1^m)$. It is known that some kinds of $Y$-algebras can be realized as a coset of the pair of  $\mathcal{W}^k(\mathfrak{sl}(m+n),f_{n,m})$ and the universal affine vertex algebra associated with $\mathfrak{gl}(m)$ up to Heisenberg algebras. In this case, Creutzig-Linshaw \cite{CR} have proved the triality conjecture. This result is the generalization of the Feigin-Frenkel duality and the coset realization of principal $W$-algebra.

By Theorem~\ref{B}, for a non-critical level, the completion of image of $\Phi_1$ coincides with the subalgebra isomorphic to the universal enveloping algebra of a subalgebra which is expected to be isomorphic to the rectangular $W$-algebra associated with $\mathfrak{gl}((n_1-n_2)b_1)$ and a nilpotent element of type $(n_1-n_2)^{b_1}$. Especially, in the case that $b_1=1,2$, this subalgebra coincides with the rectangular $W$-algebra. By Theorem~\ref{A}, for a non-critical level, we find that $\Phi_a$ becomes a homomorphism from the affine Yangian to the universal enveloping algebra of the coset of  the pair of $\mathcal{W}^k(\mathfrak{gl}(N),f)$ and this subalgebra. We expect that this result will be connected to obtain the generalization of the Gaiotto-Rapcak's triality.

\section{$W$-algebras of type $A$}
We fix some notations for vertex algebras. For a vertex algebra $V$, we denote the generating field associated with $v\in V$ by $v(z)=\displaystyle\sum_{n\in\mathbb{Z}}\limits v_{(n)}z^{-n-1}$. We also denote the OPE of $V$ by
\begin{equation*}
u(z)v(w)\sim\displaystyle\sum_{s\geq0}\limits \dfrac{(u_{(s)}v)(w)}{(z-w)^{s+1}}
\end{equation*}
for all $u, v\in V$. We denote the vacuum vector (resp.\ the translation operator) by $|0\rangle$ (resp.\ $\partial$).
We denote the universal affine vertex algebra associated with a finite dimensional Lie algebra $\mathfrak{g}$ and its inner product $\kappa$ by $V^\kappa(\mathfrak{g})$. By the PBW theorem, we can identify $V^\kappa(\mathfrak{g})$ with $U(t^{-1}\mathfrak{g}[t^{-1}])$. In order to simplify the notation, here after, we denote the generating field $(ut^{-1})(z)$ as $u(z)$ for $u\in\mathfrak{g}$. By the definition of $V^\kappa(\mathfrak{g})$, the generating fields $u(z)$ and $v(z)$ satisfy the OPE
\begin{gather}
u(z)v(w)\sim\dfrac{[u,v](w)}{z-w}+\dfrac{\kappa(u,v)}{(z-w)^2}\label{OPE1}
\end{gather}
for all $u,v\in\mathfrak{g}$.

We take a positive integer $N$ and its partition:
\begin{equation}
N=\displaystyle\sum_{i=1}^lq_i,\qquad q_1\geq q_{2}\geq\cdots\geq q_l.\label{cond:q}
\end{equation}
For $\{q_i\}_{i=1}^l$ in \eqref{cond:q}, we set $\{n_s\}_{s=1}^{v+1}$ and $\{b_s\}_{s=1}^v$ as
\begin{equation*}
q_{b_{s-1}}=n_{s-1}>q_{b_{s-1}+1}=q_{b_{s-1}+2}=\cdots=q_{b_s}=n_s>q_{b_s+1}=n_{s+1}, n_{v+1}=0.
\end{equation*}
We also fix an inner product of $\mathfrak{gl}(N)$ determined by
\begin{equation*}
(E_{i,j}|E_{p,q})=k\delta_{i,q}\delta_{p,j}+\delta_{i,j}\delta_{p,q}.
\end{equation*}
For $1\leq i\leq N$, we set $1\leq \col(i)\leq l$ and $q_1-q_{\col(i)}<\row(i)\leq q_1$ satisfying
\begin{gather*}
\col(i)=s\text{ if }\sum_{j=1}^{s-1}q_j<i\leq\sum_{i=1}^sq_j,\ 
\row(i)=i-\sum_{j=1}^{\col(i)-1}q_j+q_1-q_{\col(i)}.
\end{gather*}
We take a nilpotent element $f$ as 
\begin{equation*}
f=\sum_{1\leq j\leq N}\limits e_{\hat{j},j},
\end{equation*}
where the integer $1\leq \hat{j}\leq N$ is determined by
\begin{equation*}
\col(\hat{j})=\col(j)+1,\ \row(\hat{j})=\row(j).
\end{equation*}
The $W$-algebra $\mathcal{W}^k(\mathfrak{gl}(N),f)$ is a vertex algebra associated with a finite dimensional reductive Lie algebra $\mathfrak{g}$, a nilpotent element $f$ and a complex number $k$. It is defined by the quantum Drinfeld-Sokolov reduction (see \cite{KW1} and \cite{KW2}).
Let us set the inner product on $\mathfrak{gl}(q_s)$ by
\begin{equation*}
\kappa_s(E_{i,j},E_{p,q})=\delta_{j,p}\delta_{i,q}\alpha_s+\delta_{i,j}\delta_{p,q},
\end{equation*}
where $\alpha_s=k+N-q_s$. Then, by Corollary 5.2 in \cite{Genra}, the $W$-algebra $\mathcal{W}^k(\mathfrak{gl}(N),f)$ can be embedded into $\bigotimes_{1\leq s\leq l}V^{\kappa_s}(\mathfrak{gl}(q_s))$. This embedding is called the Miura map. 
\begin{Theorem}[Theorem 3.6 in \cite{U9}]\label{Generators}
We set $\gamma_a=\sum_{u=a+1}^{l}\limits \alpha_{u}$. For positive integers $q_1-q_{v}<p,q\leq q_1-q_{v+1}$, the folloing elemets of $\bigotimes_{1\leq s\leq l}V^{\kappa_s}(\mathfrak{gl}(q_s))$ are contained in $\mathcal{W}^k(\mathfrak{gl}(N),f)$:
\begin{align*}
W^{(1)}_{p,q}&=\sum_{\substack{1\leq r\leq l}}E^{(r)}_{p,q}[-1],\\
W^{(2)}_{p,q}&=-\sum_{\substack{1\leq r\leq l}}\gamma_{r}E^{(r)}_{p,q}[-2]+\sum_{\substack{r_1<r_2\\u>q_1-q_v}}\limits E^{(r_1)}_{u,q}[-1]E^{(r_2)}_{p,u}[-1]\\
&\quad-\sum_{\substack{r_1\geq r_2\\q_1-q_{r_2}<u\leq q_1-q_v}}\limits E^{(r_1)}_{u,q}[-1]E^{(r_2)}_{p,u}[-1],
\end{align*}
where we denote $E_{i,j}t^{-u}\in U(t^{-1}\mathfrak{gl}(q_s)[t^{-1}])=V^\kappa(\mathfrak{gl}(q_s))$ by $E_{i+(q_1-q_s),j+(q_1-q_s)}^{(s)}[-u]$.
\end{Theorem}

We denote by $\mathcal{W}^k(\mathfrak{gl}(ln),(l^n))$ the $W$-algebra $\mathcal{W}^k(\mathfrak{gl}(N),f)$ in the case that $q_1=q_2=\cdots=q_l=n$ and call a rectangular $W$-algebra of type $A$. Arakawa-Molev \cite{AM} gave strong generators of $\mathcal{W}^k(\mathfrak{gl}(ln),(l^n))$ explicitly. In the appendix, we will give more elements of $\mathcal{W}^k(\mathfrak{gl}(N),f)$:
\begin{equation*}
\{\widetilde{W}^{(r)}_{i,j}\mid1\leq i,j\leq n_1-n_2,1\leq r\leq b_1\},
\end{equation*}
which corresponds to strong generators of Arakawa-Molev in the case that $q_1=q_2=\cdots=q_l=n$. In the case that $b_1=2$, we have known the relationship between the elements $\widetilde{W}_{i,j}^{(r)}$ and $W_{i,j}^{(r)}$.
\begin{Theorem}\label{emb}
There exists an embedding
\begin{equation*}
\iota\colon\mathcal{W}^{k+n_2+\sum_{s=3}^{l}q_s}(\mathfrak{gl}(2(n_1-n_2)),(2^{n_1-n_2}))\to\mathcal{W}^k(\mathfrak{gl}(N),f)
\end{equation*}
defined by
\begin{equation*}
\widetilde{W}^{(1)}_{i,j}\mapsto W^{(1)}_{i,j},\  \widetilde{W}^{(2)}_{i,j}\mapsto W^{(2)}_{i,j}+\gamma_1\partial W^{(1)}_{i,j}.
\end{equation*}
\end{Theorem}
\begin{proof}
By the form of $W^{(r)}_{i,j}$ and $W^{(2)}_{i,j}$, we find that the OPEs of $\{W^{(r)}_{i,j}\mid r=1,2,1\leq i,j\leq n_1-n_2\}$ in $\mathcal{W}^k(\mathfrak{gl}(N),f)$ is the same as those of
\begin{equation*}
\{\widetilde{W}^{(1)}_{i,j},\widetilde{W}^{(2)}_{i,j}-\gamma_{1}\partial \widetilde{W}^{(1)}_{i,j}\mid r=1,2,1\leq i,j\leq n_1-n_2\}
\end{equation*}
in $\mathcal{W}^{k+\sum_{s=3}^{l}q_s}(\mathfrak{gl}(2n_1),(2^{n_1}))$. 
In \cite{Rap}, the OPEs of  rectangular $W$-algebras have been computed (see also \cite{U8}). Then, we find an embedding
\begin{equation*}
\iota_1\colon\mathcal{W}^{k+n_2+\sum_{s=3}^{l}q_s}(\mathfrak{gl}(2(n_1-n_2)),(2^{n_1-n_2}))\to\mathcal{W}^{k+\sum_{s=3}^{l}q_s}(\mathfrak{gl}(2n_1),(2^{n_1}))
\end{equation*}
given by $\widetilde{W}^{(r)}_{i,j}\mapsto\widetilde{W}^{(r)}_{i,j}$ for $1\leq i,j\leq n_1-n_2$. 
\end{proof}
Next, let us recall the universal enveloping algebra of a vertex algebra.
For any vertex algebra $V$, let $L(V)$ be the Borcherds Lie algebra, that is,
\begin{align}
 L(V)=V{\otimes}\mathbb{C}[t,t^{-1}]/\text{Im}(\partial\otimes\id +\id\otimes\frac{d}{d t})\label{844},
\end{align}
where the commutation relation is given by
\begin{align*}
 [ut^a,vt^b]=\sum_{r\geq 0}\begin{pmatrix} a\\r\end{pmatrix}(u_{(r)}v)t^{a+b-r}
\end{align*}
for all $u,v\in V$ and $a,b\in \mathbb{Z}$. 
\begin{Definition}[Section~6 in \cite{MNT}]\label{Defi}
We set $\mathcal{U}(V)$ as the quotient algebra of the standard degreewise completion of the universal enveloping algebra of $L(V)$ by the completion of the two-sided ideal generated by
\begin{gather}
(u_{(a)}v)t^b-\sum_{i\geq 0}
\begin{pmatrix}
 a\\i
\end{pmatrix}
(-1)^i(ut^{a-i}vt^{b+i}-(-1)^avt^{a+b-i}ut^{i}),\label{241}\\
|0\rangle t^{-1}-1.
\end{gather}
We call $\mathcal{U}(V)$ the universal enveloping algebra of $V$.
\end{Definition}
Induced by the Miura map $\mu$, we obtain the embedding
\begin{equation*}
\widetilde{\mu}\colon \mathcal{U}(\mathcal{W}^{k}(\mathfrak{gl}(N),f))\to {\widehat{\bigotimes}}_{1\leq s\leq l}U(\widehat{\mathfrak{gl}}(q_s)),
\end{equation*}
where ${\widehat{\bigotimes}}_{1\leq s\leq l}U(\widehat{\mathfrak{gl}}(q_s))$ is the standard degreewise completion of $\bigotimes_{1\leq s\leq l}U(\widehat{\mathfrak{gl}}(q_s))$.
\section{The affine Yangian and its coproduct}
The affine Yangian of type $A$ was first introduced by Guay (\cite{Gu2} and \cite{Gu1}). The affine Yangian $Y_{\hbar,\ve}(\widehat{\mathfrak{sl}}(n))$ is a 2-parameter Yangian and is the deformation of the universal enveloping algebra of the central extension of $\mathfrak{sl}(n)[u^{\pm1},v]$. The affine Yangian has several presentations. In this section, we define the affine Yangian of type $A$ by using the minimalistic presentation given in Guay-Nakajima-Wendlandt \cite{GNW}.

Hereafter, we sometimes identify $\{0,1,2,\cdots,n-1\}$ with $\mathbb{Z}/n\mathbb{Z}$. We also set$\{X,Y\}=XY+YX$ and
\begin{equation*}
a_{i,j} =\begin{cases}
2&\text{if } i=j, \\
-1&\text{if }j=i\pm 1,\\
0&\text{otherwise}
	\end{cases}
\end{equation*}
for $i\in\mathbb{Z}/n\mathbb{Z}$.
\begin{Definition}\label{Prop32}
Suppose that $n\geq3$. The affine Yangian $Y_{\hbar,\ve}(\widehat{\mathfrak{sl}}(n))$ is the associative algebra  generated by $X_{i,r}^{+}, X_{i,r}^{-}, H_{i,r}$ $(i \in \{0,1,\cdots, n-1\}, r = 0,1)$ subject to the following defining relations:
\begin{gather}
[H_{i,r}, H_{j,s}] = 0,\label{Eq2.1}\\
[X_{i,0}^{+}, X_{j,0}^{-}] = \delta_{i,j} H_{i, 0},\label{Eq2.2}\\
[X_{i,1}^{+}, X_{j,0}^{-}] = \delta_{i,j} H_{i, 1} = [X_{i,0}^{+}, X_{j,1}^{-}],\label{Eq2.3}\\
[H_{i,0}, X_{j,r}^{\pm}] = \pm a_{i,j} X_{j,r}^{\pm},\label{Eq2.4}\\
[\tilde{H}_{i,1}, X_{j,0}^{\pm}] = \pm a_{i,j}\left(X_{j,1}^{\pm}\right),\text{ if }(i,j)\neq(0,n-1),(n-1,0),\label{Eq2.5}\\
[\tilde{H}_{0,1}, X_{n-1,0}^{\pm}] = \mp \left(X_{n-1,1}^{\pm}+(\ve+\dfrac{n}{2}\hbar) X_{n-1, 0}^{\pm}\right),\label{Eq2.6}\\
[\tilde{H}_{n-1,1}, X_{0,0}^{\pm}] = \mp \left(X_{0,1}^{\pm}-(\ve+\dfrac{n}{2}\hbar) X_{0, 0}^{\pm}\right),\label{Eq2.7}\\
[X_{i, 1}^{\pm}, X_{j, 0}^{\pm}] - [X_{i, 0}^{\pm}, X_{j, 1}^{\pm}] = \pm a_{ij}\dfrac{\hbar}{2} \{X_{i, 0}^{\pm}, X_{j, 0}^{\pm}\}\text{ if }(i,j)\neq(0,n-1),(n-1,0),\label{Eq2.8}\\
[X_{0, 1}^{\pm}, X_{n-1, 0}^{\pm}] - [X_{0, 0}^{\pm}, X_{n-1, 1}^{\pm}]= \mp\dfrac{\hbar}{2} \{X_{0, 0}^{\pm}, X_{n-1, 0}^{\pm}\} + (\ve+\dfrac{n}{2}\hbar) [X_{0, 0}^{\pm}, X_{n-1, 0}^{\pm}],\label{Eq2.9}\\
(\ad X_{i,0}^{\pm})^{1+|a_{i,j}|} (X_{j,0}^{\pm})= 0 \ \text{ if }i \neq j, \label{Eq2.10}
\end{gather}
where we set $\widetilde{H}_{i,1}=H_{i,1}-\dfrac{\hbar}{2}H_{i,0}^2$.
\end{Definition}
We set an affine Lie algebra 
\begin{equation*}
\widehat{\mathfrak{gl}}(n)=\mathfrak{gl}(n)\otimes\mathbb{C}[z^{\pm1}]\oplus\mathbb{C}\tilde{c}\oplus\mathbb{C}z
\end{equation*}
whose commutator relations are given by
\begin{gather*}
\text{$z$ and $\tilde{c}$ are central elements of }\widehat{\mathfrak{gl}}(n),\\
\begin{align*}
[E_{i,j}\otimes t^u, E_{p,q}\otimes t^v]
&=(\delta_{j,p}E_{i,q}-\delta_{i,q}E_{p,j})\otimes t^{u+v}+\delta_{u+v,0}u(\delta_{i,q}\delta_{j,p}\widetilde{c}+\delta_{i,j}\delta_{p,q}z).
\end{align*}
\end{gather*}
We take Chevalley generators of $\widehat{\mathfrak{sl}}(n)$ as
\begin{gather*}
h_i=\begin{cases}
E_{n,n}-E_{1,1}+\widetilde{c}&\text{ if }i=0,\\
E_{i,i}-E_{i+1,i+1}&\text{ if }i\neq0,
\end{cases}\\
x^+_i=\begin{cases}
E_{n,1}t&\text{ if }i=0,\\
E_{i,i+1}&\text{ if }i\neq0.
\end{cases}\ x^-_i=\begin{cases}
E_{1,n}t^{-1}&\text{ if }i=0,\\
E_{i+1,i}&\text{ if }\neq0.
\end{cases}
\end{gather*}
By the defining relations \eqref{Eq2.1}-\eqref{Eq2.10}, there exists a homomorphism from the universal enveloping algebra of $\widehat{\mathfrak{sl}}(n)$ to the affine Yangian $Y_{\hbar,\ve}(\widehat{\mathfrak{sl}}(n))$ given by $h_i\mapsto H_{i,0}$ and $x^\pm_i\mapsto X^\pm_{i,0}$. We denote the image of $x\in U(\widehat{\mathfrak{sl}}(n))$ via this homomorphism by $x\in Y_{\hbar,\ve}(\widehat{\mathfrak{sl}}(n))$.

Let us set the degree of $Y_{\hbar,\ve}(\widehat{\mathfrak{sl}}(n))$ by
\begin{equation*}
\text{deg}(H_{i,r})=0,\ \text{deg}(X^\pm_{i,r})=\begin{cases}
\pm1&\text{ if }i=0,\\
0&\text{ if }i\neq0.
\end{cases}
\end{equation*}
This degree is compatible with the natural degree on $\widehat{\mathfrak{sl}}(n)$. We denote the standard degreewise completion of $Y_{\hbar,\ve}(\widehat{\mathfrak{sl}}(n))$ by $\widetilde{Y}_{\hbar,\ve}(\widehat{\mathfrak{sl}}(n))$. 
By using the minimalistic presentation of the Yangian, Guay-Nakajima-Wendlandt \cite{GNW} gave a coproduct for the Yangian associated with a Kac-Moody Lie algebra of the affine type.
\begin{Theorem}[Theorem~5.2 in \cite{GNW}]
There exists an algebra homomorphism
\begin{equation*}
\Delta^n\colon Y_{\hbar,\ve}(\widehat{\mathfrak{sl}}(n))\to Y_{\hbar,\ve}(\widehat{\mathfrak{sl}}(n))\widehat{\otimes} Y_{\hbar,\ve}(\widehat{\mathfrak{sl}}(n))
\end{equation*}
determined by
\begin{gather*}
\Delta^n(X^\pm_{j,0})=X^\pm_{j,0}\otimes1+1\otimes X^\pm_{j,0}\text{ for }0\leq j\leq n-1,\\
\Delta^n(\widetilde{H}_{i,1})=\widetilde{H}_{i,1}\otimes1+1\otimes \widetilde{H}_{i,1}+A_i\text{ for }1\leq i\leq n-1,
\end{gather*}
where $Y_{\hbar,\ve}(\widehat{\mathfrak{sl}}(n))\widehat{\otimes} Y_{\hbar,\ve}(\widehat{\mathfrak{sl}}(n))$ is the standard degreewise completion of $\otimes^2Y_{\hbar,\ve}(\widehat{\mathfrak{sl}}(n))$ and
\begin{align*}
A_i&=-\hbar(E_{i,i}\otimes E_{i+1,i+1}+E_{i+1,i+1}\otimes E_{i,i})\\
&\quad+\hbar\displaystyle\sum_{s \geq 0}  \limits\displaystyle\sum_{u=1}^{i}\limits (-E_{u,i}t^{-s-1}\otimes E_{i,u}t^{s+1}+E_{i,u}t^{-s}\otimes E_{u,i}t^s)\\
&\quad+\hbar\displaystyle\sum_{s \geq 0} \limits\displaystyle\sum_{u=i+1}^{n}\limits (-E_{u,i}t^{-s}\otimes E_{i,u}t^{s}+E_{i,u}t^{-s-1}\otimes E_{u,i}t^{s+1})\\
&\quad-\hbar\displaystyle\sum_{s \geq 0}\limits\displaystyle\sum_{u=1}^{i}\limits (-E_{u,i+1}t^{-s-1}\otimes E_{i+1,u}t^{s+1}+E_{i+1,u}t^{-s}\otimes E_{u,i+1}t^s)\\
&\quad-\hbar\displaystyle\sum_{s \geq 0}\limits\displaystyle\sum_{u=i+1}^{n} \limits (-E_{u,i+1}t^{-s}\otimes E_{i+1,u}t^{s}+E_{i+1,u}t^{-s-1}\otimes E_{u,i+1}t^{s+1}).
\end{align*}
\end{Theorem}
The homomorphism $\Delta^n$ is said to be the coproduct for the affine Yangian since $\Delta^n$ satisfies the coassociativity.
\begin{Remark}\label{remark}
By \eqref{Eq2.2}, \eqref{Eq2.3}, \eqref{Eq2.5} and \eqref{Eq2.7}, we find that the affine Yangian $Y_{\hbar,\ve}(\widehat{\mathfrak{sl}}(n))$ is generated by $X^\pm_{i,0}$ for $0\leq i\leq n-1$ and $H_{j,1}$ for $1\leq j\leq n-1$. Thus, the homomorphism is defined if the image of these elements are determined. In this article, we will define homomorphisms by giving the image of these elements.
\end{Remark}
\section{Evaluation map for the affine Yangian}
Guay \cite{Gu1} gave the evaluation map for the affine Yangian and Kodera \cite{K2} showed the surjectivity of the evaluation map.
The evaluation map for the affine Yangian is a non-trivial homomorphism from the affine Yangian $Y_{\hbar,\ve}(\widehat{\mathfrak{sl}}(n))$ to the completion of the universal enveloping algebra of the affinization of $\mathfrak{gl}(n)$.
We take the grading of $U(\widehat{\mathfrak{gl}}(n))/U(\widehat{\mathfrak{gl}}(n))(z-1)$ as $\text{deg}(Xt^s)=s$ and $\text{deg}(\tilde{c})=0$. We denote the degreewise completion of $U(\widehat{\mathfrak{gl}}(n))/U(\widehat{\mathfrak{gl}}(n))(z-1)$ by $\mathcal{U}(\widehat{\mathfrak{gl}}(n))$.
\begin{Theorem}[Theorem 3.8 in \cite{K1} and Theorem 4.18 in \cite{K2}]\label{thm:main}
\begin{enumerate}
\item
Suppose that $\hbar\neq0$ and $\tilde{c} =\dfrac{\ve}{\hbar}$.
For $a\in\mathbb{C}$, there exists an algebra homomorphism 
\begin{equation*}
\ev_{\hbar,\ve}^{n,a} \colon Y_{\hbar,\ve}(\widehat{\mathfrak{sl}}(n)) \to \mathcal{U}(\widehat{\mathfrak{gl}}(n))
\end{equation*}
uniquely determined by 
\begin{gather*}
	\ev_{\hbar,\ve}^{n,a}(X_{i,0}^{+}) = \begin{cases}
E_{n,1}t&\text{ if }i=0,\\
E_{i,i+1}&\text{ if }1\leq i\leq n-1,
\end{cases} \ev_{\hbar,\ve}^{n,a}(X_{i,0}^{-}) = \begin{cases}
E_{1,n}t^{-1}&\text{ if }i=0,\\
E_{i+1,i}&\text{ if }1\leq i\leq n-1,
\end{cases}
\end{gather*}
and
\begin{align*}
\ev_{\hbar,\ve}^{n,a}(H_{i,1}) &=(a-\dfrac{i}{2}\hbar) \ev_{\hbar,\ve}^{n,a}(H_{i,0}) -\hbar E_{i,i}E_{i+1,i+1} \\
&\quad+ \hbar \displaystyle\sum_{s \geq 0}  \limits\displaystyle\sum_{k=1}^{i}\limits  E_{i,k}t^{-s}E_{k,i}t^s+\hbar \displaystyle\sum_{s \geq 0} \limits\displaystyle\sum_{k=i+1}^{n}\limits  E_{i,k}t^{-s-1}E_{k,i}t^{s+1}\\
&\quad-\hbar\displaystyle\sum_{s \geq 0}\limits\displaystyle\sum_{k=1}^{i}\limits E_{i+1,k}t^{-s} E_{k,i+1}t^{s}-\hbar\displaystyle\sum_{s \geq 0}\limits\displaystyle\sum_{k=i+1}^{n} \limits E_{i+1,k}t^{-s-1} E_{k,i+1}t^{s+1}
\end{align*}
for $i\neq0$.

\item In the case that $\ve\neq0$, the image of $\ev_{\hbar,\ve}^{n,a}$ is dense in the target.
\end{enumerate}
\end{Theorem}
\section{Tensor product of the affine Yangians}
In this section, we will construct a homomorphism from tensor product of affine Yangians to another tensor of affine Yangians.
\begin{Theorem}[Theorem 3.1 in \cite{U8} and Theorem 3.1 in \cite{U10}]
\begin{enumerate}
    \item For $m\geq 1$ and $n\geq 3$, there exists a homomorphism
    \begin{gather*}
\Psi_1^{n,m+n}\colon Y_{\hbar,\ve}(\widehat{\mathfrak{sl}}(n))\to \widetilde{Y}_{\hbar,\ve}(\widehat{\mathfrak{sl}}(m+n))
\end{gather*}
given by
\begin{gather*}
\Psi_1^{n,m+n}(X^+_{i,0})=\begin{cases}
E_{n,1}t&\text{ if }i=0,\\
E_{i,i+1}&\text{ if }i\neq 0,
\end{cases}\ 
\Psi_1^{n,m+n}(X^-_{i,0})=\begin{cases}
E_{1,n}t^{-1}&\text{ if }i=0,\\
E_{i+1,i}&\text{ if }i\neq 0,
\end{cases}
\end{gather*}
and
\begin{align*}
\Psi_1^{n,m+n}(H_{i,1})&= H_{i,1}-\hbar\displaystyle\sum_{s \geq 0} \limits\sum_{k=n+1}^{m+n}\limits E_{i,k}t^{-s-1} E_{k,i}t^{s+1}+\hbar\displaystyle\sum_{s \geq 0}\limits \sum_{k=n+1}^{m+n}\limits E_{i+1,k}t^{-s-1} E_{k,i+1}t^{s+1}
\end{align*}
for $i\neq0$. 
    \item For $m\geq 3$ and $n\geq 1$, there exists a homomorphism
    \begin{gather*}
\Psi_2^{m,m+n}\colon Y_{\hbar,\ve+n\hbar}(\widehat{\mathfrak{sl}}(m))\to \widetilde{Y}_{\hbar,\ve}(\widehat{\mathfrak{sl}}(m+n))
\end{gather*}
determined by
\begin{gather*}
\Psi_2^{m,m+n}(X^+_{i,0})=\begin{cases}
E_{m+n,n+1}t&\text{ if }i=0,\\
E_{n+i,n+i+1}&\text{ if }i\neq 0,
\end{cases}\ 
\Psi_2^{m,m+n}(X^-_{i,0})=\begin{cases}
E_{n+1,m+n}t^{-1}&\text{ if }i=0,\\
E_{n+i+1,n+i}&\text{ if }i\neq 0,
\end{cases}
\end{gather*}
and
\begin{align*}
\Psi_2^{m,m+n}(H_{i,1})&= H_{i+n,1}+\hbar\displaystyle\sum_{s \geq 0}\limits\sum_{k=1}^n E_{k,n+i}t^{-s-1}E_{n+i,k}t^{s+1}\\
&\quad-\hbar\displaystyle\sum_{s \geq 0}\limits\sum_{k=1}^n E_{k,n+i+1}t^{-s-1} E_{n+i+1,k}t^{s+1}
\end{align*}
for $i\neq0$.
    \item For $m,n\geq 3$, two homomorphisms $\Psi_1$ and $\Psi_2$ induce the homomorphism
    \begin{equation*}
    \Psi_1^{n,m+n}\otimes\Psi_2^{m,m+n}\colon Y_{\hbar,\ve}(\widehat{\mathfrak{sl}}(n))\otimes Y_{\hbar,\ve+n\hbar}(\widehat{\mathfrak{sl}}(m))\to \widetilde{Y}_{\hbar,\ve}(\widehat{\mathfrak{sl}}(m+n)).
    \end{equation*}
\end{enumerate}
\end{Theorem}
Suppose that $q_s-q_{s+1}=0\text{ or }\geq3$.
Let us define the following algebra
\begin{align*}
Y^{n_1,n_2,\cdots,n_v}_{\hbar,\ve}=\bigotimes_{1\leq s\leq v}Y_{\hbar,\ve_s}(\widehat{\mathfrak{sl}}(n_s-n_{s+1})),
\end{align*}
where we set $\ve_s=\ve-(n_s-n_v)\hbar$. We denote $x\in Y_{\hbar,\ve_s}(\widehat{\mathfrak{sl}}(n_s-n_{s+1}))\subset Y^{n_1,n_2,\cdots,n_v}$ by $x^{(s)}$. Then, we can construct two homomorphisms $\Delta_A^{n_1,\cdots,n_{v}}$ and $\Delta_B^{n_1,\cdots,n_{v}}$:
\begin{equation*}
\Delta_A^{n_1,\cdots,n_{v}}\colon Y^{n_1,n_2,\cdots,n_v}\to Y^{n_1,n_2,\cdots,n_v}\widehat{\otimes} Y_{\hbar,\ve}(\widehat{\mathfrak{sl}}(n_v))
\end{equation*}
determined by
\begin{align*}
\Delta_A^{n_1,\cdots,n_{v}}(x^{(r)})=\begin{cases}
x^{(r)}\otimes 1&\text{ if }r\neq v,\\
1^{\otimes(v-1)}\otimes\Delta^{n_v}(x)&\text{ if }r=v
\end{cases}
\end{align*}
and
\begin{equation*}
\Delta_B^{n_1,\cdots,n_{v}}\colon Y^{n_1,n_2,\cdots,n_l}\to Y^{n_1,n_2,\cdots,n_{v-2},n_{v-1}}\widehat{\otimes} Y_{\hbar,\ve}(\widehat{\mathfrak{sl}}(n_v))
\end{equation*}
determined by
\begin{align*}
\Delta_B^{n_1,\cdots,n_{v}}(x^{(r)})=\begin{cases}
x^{(r)}\otimes 1&\text{ if }r\neq v-2,\\
1^{\otimes(v-2)}\otimes(\Psi^{n_{v-1}-n_v,n_{v-1}}_1(x))\otimes 1&\text{ if }r=v-1,\\
1^{\otimes(v-2)}\otimes((\Psi_2^{n_v,n_{v-1}}\otimes1)\circ\Delta)(x)&\text{ if }r=v,
\end{cases}
\end{align*}
where $S\widehat{\otimes}U$ means the standard degreewise completion of the tensor product $S\otimes U$.

We set a homomorphism
\begin{align*}
\Delta_l&=\prod_{s=1}^{l-1}(\Delta_s\otimes\id^{l-s-1})\colon Y^{n_1,n_2,\cdots,n_v}_{\hbar,\ve}\to\widehat{\bigotimes}_{1\leq s\leq l}Y_{\hbar,\ve-(n_s-n_v)\hbar}(\widehat{\mathfrak{sl}}(q_s)),
\end{align*}
where $\widehat{\bigotimes}_{1\leq s\leq l}Y_{\hbar,\ve-(n_s-n_v)\hbar}(\widehat{\mathfrak{sl}}(q_s))$ is the standard degreewise completion of the tensor of affine Yangians and
\begin{align*}
\Delta_s&=\begin{cases}
\Delta_B^{n_1,\cdots,n_{a}}&\text{ if }n_{a-1}=q_{s-1}>q_s=n_a,\\
\Delta_A^{n_1,\cdots,n_{a}}&\text{ if }q_{s-1}=q_s=n_a.
\end{cases}
\end{align*}
By the definition of $\Psi^{n,m+n}$ and $\Delta^n$, we obtain
\begin{align*}
\Delta_l(H_{i,1})^{(a)}&=\sum_{1\leq s\leq b_a}H^{(s)}_{i+q_s-n_a}+B_i+C_i+D_i,
\end{align*}
where we set $H^{(s)}_{i+q_s-n_a}=1^{\otimes s-1}\otimes H_{i+q_s-n_a}\otimes 1^{\otimes l-s}$, $E^{(s)}_{q_s-n_a+i,q_s-n_a+j}t^s=1^{\otimes s-1}\otimes E_{i,j}\otimes 1^{\otimes l-s}$ and
\begin{align*}
B_i&=-\hbar\sum_{1\leq r_1<r_2\leq b_a}\limits E^{(r_1)}_{q_{r_1}-n_a+i,q_{r_1}-n_a+i}E^{(r_2)}_{q_{r_2}-n_a+i+1,q_{r_2}-n_a+i+1}\\
&\quad-\hbar\sum_{1\leq r_1<r_2\leq b_a}\limits E^{(r_1)}_{q_{r_1}-n_a+i+1,q_{r_1}-n_a+i+1}E^{(r_2)}_{q_{r_2}-n_a+i,q_{r_2}-n_a+i}\\
&\quad-\hbar\sum_{1\leq r_1<r_2\leq b_a}\limits\displaystyle\sum_{s \geq 0}  \limits\displaystyle\sum_{u=1}^{q_{r_2}-n_a+i}\limits E^{(r_1)}_{q_{r_1}-q_{r_2}+u,q_{r_1}-n_a+i}t^{-s-1}E^{(r_2)}_{q_{r_2}-n_a+i,u}t^{s+1}\\
&\quad+\hbar\sum_{1\leq r_1<r_2\leq b_a}\limits\displaystyle\sum_{s \geq 0}  \limits\displaystyle\sum_{u=1}^{q_{r_2}-n_a+i}\limits E^{(r_1)}_{q_{r_1}-n_a+i,q_{r_1}-q_{r_2}+u}t^{-s}E^{(r_2)}_{u,q_{r_2}-n_a+i}t^s\\
&\quad-\hbar\sum_{1\leq r_1<r_2\leq b_a}\limits\displaystyle\sum_{s \geq 0} \limits\displaystyle\sum_{u=q_{r_2}-n_a+i+1}^{q_{r_2}}\limits E^{(r_1)}_{q_{r_1}-q_{r_2}+u,q_{r_1}-n_a+i}t^{-s}E^{(r_2)}_{q_{r_2}-n_a+i,u}t^{s}\\
&\quad+\hbar\sum_{1\leq r_1<r_2\leq b_a}\limits\displaystyle\sum_{s \geq 0} \limits\displaystyle\sum_{u=q_{r_2}-n_a+i+1}^{q_{r_2}}\limits E^{(r_1)}_{q_{r_1}-n_a+i,q_{r_1}-q_{r_2}+u}t^{-s-1}E^{(r_2)}_{u,q_{r_2}-n_a+i}t^{s+1}\\
&\quad+\hbar\sum_{1\leq r_1<r_2\leq b_a}\limits\displaystyle\sum_{s \geq 0}\limits\displaystyle\sum_{u=1}^{q_{r_2}-n_a+i}\limits E^{(r_1)}_{q_{r_1}-q_{r_2}+u,q_{r_1}-n_a+i+1}t^{-s-1}E^{(r_2)}_{q_{r_2}-n_a+i+1,u}t^{s+1}\\
&\quad-\hbar\sum_{1\leq r_1<r_2\leq b_a}\limits\displaystyle\sum_{s \geq 0}\limits\displaystyle\sum_{u=1}^{q_{r_2}-n_a+i}\limits E^{(r_1)}_{q_{r_1}-n_a+i+1,q_{r_1}-q_{r_2}+u}t^{-s}E^{(r_2)}_{u,q_{r_2}-n_a+i+1}t^s\\
&\quad+\hbar\sum_{1\leq r_1<r_2\leq b_a}\limits\displaystyle\sum_{s \geq 0}\limits\displaystyle\sum_{u=q_{r_2}-n_a+i+1}^{q_{r_2}} \limits E^{(r_1)}_{q_{r_1}-q_{r_2}+u,q_{r_1}-q_{l}+i+1}t^{-s} E^{(r_2)}_{q_{r_2}-n_a+i+1,u}t^{s}\\
&\quad-\hbar\sum_{1\leq r_1<r_2\leq b_a}\limits\displaystyle\sum_{s \geq 0}\limits\displaystyle\sum_{u=q_{r_2}-n_a+i+1}^{q_{r_2}} \limits E^{(r_1)}_{q_{r_1}-n_a+i+1,q_{r_1}-q_{r_2}+u}t^{-s-1}E^{(r_2)}_{u,q_{r_2}-n_a+i+1}t^{s+1},\\
C_i&=\hbar\sum_{r=1}^{b_a}\displaystyle\sum_{s \geq 0}\limits\sum_{u=1}^{q_r-n_a} E^{(r)}_{u,q_r-n_a+i}t^{-s-1}E^{(r)}_{q_r-n_a+i,u}t^{s+1}\\
&\quad-\hbar\sum_{r=1}^{b_a}\displaystyle\sum_{s \geq 0}\limits\sum_{u=1}^{q_r-n_a} E^{(r)}_{u,q_r-n_a+i+1}t^{-s-1} E^{(r)}_{q_r-n_a+i+1,u}t^{s+1}\\
&\quad+\hbar\displaystyle\sum_{1\leq r_1<r_2\leq b_a}\sum_{s \geq 0}\limits\sum_{u=1}^{q_{r_2}-n_a} E^{(r_1)}_{q_{r_1}-q_{r_2}+u,q_{r_1}-n_a+i}t^{-s-1}E^{(r_2)}_{q_{r_2}-n_a+i,u}t^{s+1}\\
&\quad-\hbar\sum_{1\leq r_1<r_2\leq b_a}\displaystyle\sum_{s \geq 0}\limits\sum_{u=1}^{q_{r_2}-n_a} E^{(r_1)}_{q_{r_1}-q_{r_2}+u,q_{r_1}-n_a+i+1}t^{-s-1} E^{(r_2)}_{q_{r_2}-n_a+i+1,u}t^{s+1}\\
&\quad+\hbar\sum_{1\leq r_1<r_2\leq b_a}\displaystyle\sum_{s \geq 0}\limits\sum_{u=1}^{q_{r_2}-n_a} E^{(r_2)}_{u,q_{r_2}-n_a+i}t^{-s-1} E^{(r_1)}_{q_{r_1}-n_a+i,q_{r_1}-q_{r_2}+u}t^{s+1}\\
&\quad-\hbar\sum_{1\leq r_1<r_2\leq b_a}\displaystyle\sum_{s \geq 0}\limits\sum_{u=1}^{q_{r_2}-n_a} E^{(r_2)}_{u,q_{r_2}-n_a+i+1}t^{-s-1} E^{(r_1)}_{q_{r_1}-n_a+i+1,q_{r_1}-q_{r_2}+u}t^{s+1},\\
D_i&=-\hbar\sum_{r=1}^{b_a}\displaystyle\sum_{s \geq 0}\limits\sum_{u=q_r-n_{a+1}+1}^{q_r} E^{(r)}_{q_r-n_a+i,u}t^{-s-1}E^{(r)}_{u,q_r-n_a+i}t^{s+1}\\
&\quad+\hbar\sum_{r=1}^{b_a}\displaystyle\sum_{s \geq 0}\limits\sum_{u=q_r-n_{a+1}+1}^{q_r} E^{(r)}_{q_r-n_a+i+1,u}t^{-s-1} E^{(r)}_{u,q_r-n_a+i+1}t^{s+1}\\
&\quad-\hbar\displaystyle\sum_{1\leq r_1<r_2\leq b_a}\sum_{s \geq 0}\limits\sum_{u=q_{r_2}-n_{a+1}+1}^{q_{r_2}} E^{(r_1)}_{q_{r_1}-n_a+i,q_{r_1}-q_{r_2}+u}t^{-s-1}E^{(r_2)}_{u,q_{r_2}-n_a+i}t^{s+1}\\
&\quad+\hbar\sum_{1\leq r_1<r_2\leq b_a}\displaystyle\sum_{s \geq 0}\limits\sum_{u=q_{r_2}-n_{a+1}+1}^{q_{r_2}} E^{(r_1)}_{q_{r_1}-n_a+i+1,q_{r_1}-q_{r_2}+u}t^{-s-1}E^{(r_2)}_{u,q_{r_2}-n_a+i+1}t^{s+1}\\
&\quad-\hbar\displaystyle\sum_{1\leq r_1<r_2\leq b_a}\sum_{s \geq 0}\limits\sum_{u=q_{r_2}-n_{a+1}+1}^{q_{r_2}} E^{(r_2)}_{q_{r_2}-n_a+i,u}t^{-s-1}E^{(r_1)}_{q_{r_1}-q_{r_2}+u,q_{r_1}-n_a+i}t^{s+1}\\
&\quad+\hbar\sum_{1\leq r_1<r_2\leq b_a}\displaystyle\sum_{s \geq 0}\limits\sum_{u=q_{r_2}-n_{a+1}+1}^{q_{r_2}} E^{(r_2)}_{q_{r_2}-n_a+i+1,u}t^{-s-1}E^{(r_1)}_{q_{r_1}-q_{r_2}+u,q_{r_1}-n_a+i+1}t^{s+1}.
\end{align*}
We note that $B_i$, $C_i$ and $D_i$ come from the coproduct for the affine Yangian, the homomorphism $\Psi_2^{m,m+n}$ and the homomorphism $\Psi_1^{n,m+n}$ respectively.
\section{Tensor product of affine Yangians and non-rectangular $W$-algebras}
In this section, we construct a homomorphism from the tensor product of affine Yangians to the universal enveloping algebra of the non-rectangular $W$-algerba.
\begin{Theorem}\label{Main}
Assume that $\ve=\hbar(k+N-q_l)$. Then, there exists an algebra homomorphism 
\begin{equation*}
\Phi\colon Y^{n_1,\cdots,n_v}_{\hbar,\ve}\to \mathcal{U}(\mathcal{W}^{k}(\mathfrak{gl}(N),f))
\end{equation*} 
determined by
\begin{gather*}
\Phi((X^+_{i,0})^{(a)})=\begin{cases}
W^{(1)}_{n_1-n_{a+1},n_1-n_a+1}t&\text{ if }i=0,\\
W^{(1)}_{n_1-n_a+i,n_1-n_a+i+1}&\text{ if }1\leq i\leq n_a-n_{a+1}-1,
\end{cases}\\
\Phi((X^-_{i,0})^{(a)})=\begin{cases}
W^{(1)}_{n_1-n_a+1,n_1-n_{a+1}}t^{-1}&\text{ if }i=0,\\
W^{(1)}_{n_1-n_a+i+1,n_1-n_a+i}&\text{ if }1\leq i\leq n_a-n_{a+1}-1,
\end{cases}
\end{gather*}
and
\begin{align*}
\Phi((H_{i,1})^{(a)})&=
-\hbar(W^{(2)}_{n_1-n_a+i,n_1-n_a+i}t-W^{(2)}_{n_1-n_a+i+1,n_1-n_a+i+1}t)\\
&\quad-\dfrac{i}{2}\hbar(W^{(1)}_{n_1-n_a+i,n_1-n_a+i}-W^{(1)}_{n_1-n_a+i+1,n_1-n_a+i+1})\\
&\quad+\hbar W^{(1)}_{n_1-n_a+i,n_1-n_a+i}W^{(1)}_{n_1-n_a+i+1,n_1-n_a+i+1}\\
&\quad+\hbar\displaystyle\sum_{s \geq 0}  \limits\displaystyle\sum_{u=1}^{i}\limits W^{(1)}_{n_1-n_a+i,n_1-n_a+u}t^{-s}W^{(1)}_{n_1-n_a+u,n_1-n_a+i}t^s\\
&\quad+\hbar\displaystyle\sum_{s \geq 0} \limits\displaystyle\sum_{u=i+1}^{n}\limits W^{(1)}_{n_1-n_a+i,n_1-n_a+u}t^{-s-1} W^{(1)}_{n_1-n_a+u,n_1-n_a+i}t^{s+1}\\
&\quad-\hbar\displaystyle\sum_{s \geq 0}\limits\displaystyle\sum_{u=1}^{i}\limits W^{(1)}_{n_1-n_a+i+1,n_1-n_a+u}t^{-s} W^{(1)}_{n_1-n_a+u,n_1-n_a+i+1}t^s\\
&\quad-\hbar\displaystyle\sum_{s \geq 0}\limits\displaystyle\sum_{u=i+1}^{n} \limits W^{(1)}_{n_1-n_a+i+1,n_1-n_a+u}t^{-s-1} W^{(1)}_{n_1-n_a+u,n_1-n_a+i+1}t^{s+1}
\end{align*}
for $1\leq i\leq n_a-n_{a+1}-1$.
\end{Theorem}
Similarly to the main theorem in \cite{KU} and \cite{U11}, it is enough to show the following theorem.
\begin{Theorem}
We obtain the following relation:
\begin{equation*}
\bigotimes_{1\leq s\leq l}\ev_{\hbar,\ve-(q_s-q_l)\hbar}^{q_s,\gamma_s\hbar}\circ\Delta_l=\widetilde{\mu}\circ\Phi.
\end{equation*}
\end{Theorem}
\begin{proof}
It is enough to show that the relation holds for $X^\pm_{j,0} (0\leq j\leq n_a-n_{a+1}-1)$ and $H_{i,1} (1\leq i\leq n_a-n_{a+1}-1)$. It is trivial that the relation holds for $X^\pm_{j,0}$. Hence, we only show the latter one, that is,
\begin{equation}
\bigotimes_{1\leq s\leq l}\ev_{\hbar,\ve-(q_s-q_l)\hbar}^{q_s,(\gamma_s-\frac{q_s-n_a}{2})\hbar}\circ\Delta_l(H_{i,1})=\widetilde{\mu}\circ\Phi(H_{i,1})\label{goal}
\end{equation}
We denote $E_{i,j}^{(r)}t^a\in\widehat{\bigotimes}_{1\leq s\leq l}U(\widehat{\mathfrak{gl}}(q_s))$ by $e^{(r)}_{q_1-q_s+i,q_1-q_s+j}t^a$.

By the definition of the evaluation map and the last of Section~5, we find that the left hand side of \eqref{goal} is equal to the sum of the following four terms:
\begin{align}
&\quad\bigotimes_{1\leq s\leq b_a}\ev_{\hbar,\ve-(q_s-q_l)\hbar}^{q_s,(\gamma_s-\frac{q_s-n_a}{2})\hbar}(\bigotimes_{1\leq s\leq l}H^{(s)}_{i+q_s-n_a})\nonumber\\
&=\sum_{r=1}^{b_a}\limits(\gamma_r-\dfrac{i+q_s-n_a}{2}\hbar)(e^{(r)}_{q_1-n_a+i,q_1-n_a+i} -e^{(r)}_{q_1-n_a+i+1,q_1-n_a+i+1})\nonumber\\
&\quad-\sum_{r=1}^{b_a}\limits\hbar e^{(r)}_{q_1-n_a+i,q_1-n_a+i}e^{(r)}_{q_1-n_a+i+1,q_1-n_a+i+1} \nonumber\\
&\quad+ \hbar \sum_{r=1}^{b_a}\limits\displaystyle\sum_{s \geq 0}  \limits\displaystyle\sum_{u=q_1-q_r+1}^{q_1-n_a+i}\limits  e^{(r)}_{q_1-n_a+i,u}t^{-s}e^{(r)}_{u,q_1-n_a+i}t^s\nonumber\\
&\quad+\hbar\sum_{r=1}^{b_a}\limits \displaystyle\sum_{s \geq 0} \limits\displaystyle\sum_{u=q_1-n_a+i+1}^{q_1}\limits  e^{(r)}_{q_1-n_a+i,u}t^{-s-1}e^{(r)}_{u,q_1-n_a+i}t^{s+1}\nonumber\\
&\quad-\hbar\sum_{r=1}^{b_a}\limits\displaystyle\sum_{s \geq 0}\limits\displaystyle\sum_{u=q_1-q_r+1}^{q_1-n_a+i}\limits e^{(r)}_{q_1-n_a+i+1,u}t^{-s} e^{(r)}_{u,q_1-n_a+i+1}t^{s}\nonumber\\
&\quad-\hbar\sum_{r=1}^{b_a}\limits\displaystyle\sum_{s \geq 0}\limits\displaystyle\sum_{u=q_1-n_a+i+1}^{q_1} \limits e^{(r)}_{q_1-n_a+i+1,u}t^{-s-1} e^{(r)}_{u,q_1-n_a+i+1}t^{s+1},\label{551}\\
&\quad\bigotimes_{1\leq s\leq b_a}\ev_{\hbar,\ve-(q_s-q_l)\hbar}^{q_s,(\gamma_s-\frac{q_s-n_a}{2})\hbar}(B_i)\nonumber\\
&=-\hbar\sum_{1\leq r_1<r_2\leq b_a}\limits e^{(r_1)}_{q_1-n_a+i,q_1-n_a+i}e^{(r_2)}_{q_1-n_a+i+1,q_1-n_a+i+1}\nonumber\\
&\quad-\hbar\sum_{1\leq r_1<r_2\leq b_a}\limits e^{(r_1)}_{q_1-n_a+i+1,q_1-n_a+i+1}e^{(r_2)}_{q_1-n_a+i,q_1-n_a+i}\nonumber\\
&\quad-\hbar\sum_{1\leq r_1<r_2\leq b_a}\limits\displaystyle\sum_{s \geq 0}  \limits\displaystyle\sum_{u=1}^{q_{r_2}-n_a+i}\limits e^{(r_1)}_{q_{1}-q_{r_2}+u,q_{1}-n_a+i}t^{-s-1}E^{(r_2)}_{q_{1}-q_{r_2}+i,q_{1}-q_{r_2}+u}t^{s+1}\nonumber\\
&\quad+\hbar\sum_{1\leq r_1<r_2\leq b_a}\limits\displaystyle\sum_{s \geq 0}  \limits\displaystyle\sum_{u=1}^{q_{r_2}-n_a+i}\limits e^{(r_1)}_{q_{1}-n_a+i,q_{1}-q_{r_2}+u}t^{-s}e^{(r_2)}_{q_{1}-q_{r_2}+u,q_{1}-n_a+i}t^s\nonumber\\
&\quad-\hbar\sum_{1\leq r_1<r_2\leq b_a}\limits\displaystyle\sum_{s \geq 0} \limits\displaystyle\sum_{u=q_{r_2}-n_a+i+1}^{q_{r_2}}\limits e^{(r_1)}_{q_{1}-q_{r_2}+u,q_{1}-n_a+i}t^{-s}e^{(r_2)}_{q_{1}-n_a+i,q_{1}-q_{r_2}+u}t^{s}\nonumber\\
&\quad+\hbar\sum_{1\leq r_1<r_2\leq b_a}\limits\displaystyle\sum_{s \geq 0} \limits\displaystyle\sum_{u=q_{r_2}-n_a+i+1}^{q_{r_2}}\limits e^{(r_1)}_{q_{1}-n_a+i,q_{1}-q_{r_2}+u}t^{-s-1}e^{(r_2)}_{q_{1}-q_{r_2}+u,q_{1}-n_a+i}t^{s+1}\nonumber\\
&\quad+\hbar\sum_{1\leq r_1<r_2\leq b_a}\limits\displaystyle\sum_{s \geq 0}\limits\displaystyle\sum_{u=1}^{q_{r_2}-n_a+i}\limits e^{(r_1)}_{q_{1}-q_{r_2}+u,q_{1}-n_a+i+1}t^{-s-1}E^{(r_2)}_{q_{1}-n_a+i+1,q_{1}-q_{r_2}+u}t^{s+1}\nonumber\\
&\quad-\hbar\sum_{1\leq r_1<r_2\leq b_a}\limits\displaystyle\sum_{s \geq 0}\limits\displaystyle\sum_{u=1}^{q_{r_2}-n_a+i}\limits e^{(r_1)}_{q_{1}-n_a+i+1,q_{1}-q_{r_2}+u}t^{-s}e^{(r_2)}_{q_{1}-q_{r_2}+u,q_{1}-n_a+i+1}t^s\nonumber\\
&\quad+\hbar\sum_{1\leq r_1<r_2\leq b_a}\limits\displaystyle\sum_{s \geq 0}\limits\displaystyle\sum_{u=q_{r_2}-n_a+i+1}^{q_{r_2}} \limits e^{(r_1)}_{q_{1}-q_{r_2}+u,q_{1}-n_a+i+1}t^{-s} e^{(r_2)}_{q_{1}-n_a+i+1,q_{1}-q_{r_2}+u}t^{s}\nonumber\\
&\quad-\hbar\sum_{1\leq r_1<r_2\leq b_a}\limits\displaystyle\sum_{s \geq 0}\limits\displaystyle\sum_{u=q_{r_2}-n_a+i+1}^{q_{r_2}} \limits e^{(r_1)}_{q_{1}-n_a+i+1,q_{1}-q_{r_2}+u}t^{-s-1}e^{(r_2)}_{q_{1}-q_{r_2}+u,q_{1}-n_a+i+1}t^{s+1},\label{552}\\
&\quad\bigotimes_{1\leq s\leq b_a}\ev_{\hbar,\ve-(q_s-q_l)\hbar}^{q_s,(\gamma_s-\frac{q_s-n_a}{2})\hbar}(C_i)\nonumber\\
&=\hbar\sum_{r=1}^{b_a}\displaystyle\sum_{s \geq 0}\limits\sum_{u=1}^{q_r-n_a} e^{(r)}_{q_1-q_r+u,q_1-n_a+i}t^{-s-1}e^{(r)}_{q_1-q_r+i,q_1-n_a+u}t^{s+1}\nonumber\\
&\quad-\hbar\sum_{r=1}^{b_a}\displaystyle\sum_{s \geq 0}\limits\sum_{u=1}^{q_r-n_a} e^{(r)}_{q_1-q_r+u,q_1-q_r+i+1}t^{-s-1} e^{(r)}_{q_1-q_r+i+1,q_1-q_r+u}t^{s+1}\nonumber\\
&\quad+\hbar\displaystyle\sum_{1\leq r_1<r_2\leq b_a}\sum_{s \geq 0}\limits\sum_{u=1}^{q_{r_2}-n_a} e^{(r_1)}_{q_1-q_{r_2}+u,q_1-n_a+i}t^{-s-1}e^{(r_2)}_{q_1-n_a+i,q_1-q_{r_2}+u}t^{s+1}\nonumber\\
&\quad-\hbar\displaystyle\sum_{s \geq 0}\limits\sum_{u=1}^{q_{r_2}-n_a} e^{(r_1)}_{q_1-q_{r_2}+u,q_1-n_a+i+1}t^{-s-1} e^{(r_2)}_{q_1-n_a+i+1,q_1-q_{r_2}+u}t^{s+1}\nonumber\\
&\quad+\hbar\displaystyle\sum_{s \geq 0}\limits\sum_{u=1}^{q_{r_2}-n_a} e^{(r_2)}_{q_1-q_{r_2}+u,q_1-n_a+i}t^{-s-1} e^{(r_1)}_{q_1-n_a+i,q_1-q_{r_2}+u}t^{s+1}\nonumber\\
&\quad-\hbar\displaystyle\sum_{s \geq 0}\limits\sum_{u=1}^{q_{r_2}-n_a} e^{(r_2)}_{q_1-q_{r_2}+u,q_1-n_a+i+1}t^{-s-1} e^{(r_1)}_{q_1-n_a+i+1,q_1-q_{r_2}+u}t^{s+1},\label{553}\\
&\quad\bigotimes_{1\leq s\leq l}\ev_{\hbar,\ve-(q_s-q_l)\hbar}^{q_s,(\gamma_s-\frac{q_s-n_a}{2})\hbar}(D_i)\nonumber\\
&=-\hbar\sum_{r=1}^{b_a}\displaystyle\sum_{s \geq 0}\limits\sum_{u=q_r-n_{a+1}+1}^{q_r} e^{(r)}_{q_1-n_a+i,q_1-q_r+u}t^{-s-1}e^{(r)}_{q_1-q_r+u,q_1-n_a+i}t^{s+1}\nonumber\\
&\quad+\hbar\sum_{r=1}^{b_a}\displaystyle\sum_{s \geq 0}\limits\sum_{u=q_r-n_{a+1}+1}^{q_r} e^{(r)}_{q_1-n_a+i+1,q_1-q_r+u}t^{-s-1} e^{(r)}_{q_1-q_r+u,q_1-n_a+i+1}t^{s+1}\nonumber\\
&\quad-\hbar\displaystyle\sum_{1\leq r_1<r_2\leq b_a}\sum_{s \geq 0}\limits\sum_{u=q_{r_2}-n_{a+1}+1}^{q_{r_2}} e^{(r_1)}_{q_{1}-n_a+i,q_{1}-q_{r_2}+u}t^{-s-1}e^{(r_2)}_{q_{1}-q_{r_2}+u,q_{1}-n_a+i}t^{s+1}\nonumber\\
&\quad+\hbar\sum_{1\leq r_1<r_2\leq b_a}\displaystyle\sum_{s \geq 0}\limits\sum_{u=q_{r_2}-n_{a+1}+1}^{q_{r_2}} e^{(r_1)}_{q_{1}-n_a+i+1,q_{1}-q_{r_2}+u}t^{-s-1}e^{(r_2)}_{q_{1}-q_{r_2}+u,q_{1}-n_a+i+1}t^{s+1}\nonumber\\
&\quad-\hbar\displaystyle\sum_{1\leq r_1<r_2\leq b_a}\sum_{s \geq 0}\limits\sum_{u=q_{r_2}-n_{a+1}+1}^{q_{r_2}} e^{(r_2)}_{q_1-n_a+i,q_{1}-q_{r_2}+u}t^{-s-1}e^{(r_1)}_{q_{1}-q_{r_2}+u,q_{1}-n_a+i}t^{s+1}\nonumber\\
&\quad+\hbar\sum_{1\leq r_1<r_2\leq b_a}\displaystyle\sum_{s \geq 0}\limits\sum_{u=q_{r_2}-n_{a+1}+1}^{q_{r_2}} e^{(r_2)}_{q_{1}-n_a+i+1,q_{1}-q_{r_2}+u}t^{-s-1}e^{(r_1)}_{q_{1}-q_{r_2}+u,q_{1}-n_a+i+1}t^{s+1}.\label{555}
\end{align}
By the definition of $\widetilde{\mu}$, we have
\begin{align}
&\quad\hbar\widetilde{\mu}(W^{(2)}_{q_1-n_a+i,q_1-n_a+i}t)-\hbar\widetilde{\mu}(W^{(2)}_{q_1-n_a+i+1,q_1-n_a+i+1}t)\nonumber\\
&=\hbar\sum_{r=1}^{b_a}\gamma_re^{(r)}_{q_1-n_a+i,q_1-n_a+i}+\sum_{s\in\mathbb{Z}}\limits\sum_{1\leq r_1<r_2\leq b_a}\sum_{u>q_1-n_a}\limits e^{(r_1)}_{u,q_1-n_a+i}t^{-s}e^{(r_2)}_{q_1-n_a+i,u}t^s\nonumber\\
&\quad-\hbar\sum_{s\geq0}\limits\sum_{r=1}^{b_a}\limits\sum_{1\leq u \leq q_r-n_a}\limits e^{(r)}_{q_1-q_r+u,q_1-n_a+i}t^{-s-1}e^{(r)}_{q_1-n_a+i,q_1-q_r+u}t^{s+1}\nonumber\\
&\quad-\hbar\sum_{s\geq0}\limits\sum_{r=1}^{b_a}\limits\sum_{1\leq u \leq q_r-n_a}\limits e^{(r)}_{q_1-n_a+i,q_1-q_r+u}t^{-s}e^{(r)}_{q_1-q_r+u,q_1-n_a+i}t^{s}\nonumber\\
&\quad-\hbar\sum_{s\in\mathbb{Z}}\limits\sum_{1\leq r_1<r_2\leq b_a}\limits\sum_{1\leq u \leq q_{r_2}-n_a}\limits e^{(r_1)}_{q_1-n_a+i,q_1-q_{r_2}+u}t^{s+1}e^{(r_2)}_{q_1-q_{r_2}+u,q_1-n_a+i}t^{-s-1}\nonumber\\
&\quad-\hbar\sum_{r=1}^{b_a}\gamma_re^{(r)}_{q_1-n_a+i+1,q_1-n_a+i+1}-\sum_{s\in\mathbb{Z}}\limits\sum_{1\leq r_1<r_2\leq b_a}\sum_{u>q_1-n_a}\limits e^{(r_1)}_{u,q_1-n_a+i+1}t^{-s}e^{(r_2)}_{q_1-n_a+i+1,u}t^s\nonumber\\
&\quad+\hbar\sum_{s\geq0}\limits\sum_{r=1}^{b_a}\limits\sum_{1\leq u \leq q_r-n_a}\limits e^{(r)}_{q_1-q_r+u,q_1-n_a+i+1}t^{-s-1}e^{(r)}_{q_1-n_a+i+1,q_1-q_r+u}t^{s+1}\nonumber\\
&\quad+\hbar\sum_{s\geq0}\limits\sum_{r=1}^{b_a}\limits\sum_{1\leq u \leq q_r-n_a}\limits e^{(r)}_{q_1-n_a+i+1,q_1-q_r+u}t^{-s}e^{(r)}_{q_1-q_r+u,q_1-n_a+i+1}t^{s}\nonumber\\
&\quad+\hbar\sum_{s\in\mathbb{Z}}\limits\sum_{1\leq r_1<r_2\leq b_a}\limits\sum_{1\leq u \leq q_{r_2}-n_a}\limits e^{(r_1)}_{q_1-n_a+i+1,q_1-q_{r_2}+u}t^{s+1}e^{(r_2)}_{q_1-q_{r_2}+u,q_1-n_a+i+1}t^{-s-1}.\label{554}
\end{align}
Here after, in order to simplify the notation, we denote the right hand side of $i$-th term of equation labeled $(\ )$ by $(\ )_i$.
The sum of \eqref{551}-\eqref{554} is equal to the sum of the left hand side of \eqref{goal} and $\hbar\widetilde{\mu}(W^{(2)}_{q_1-n_a+i,q_1-n_a+i}t)-\hbar\widetilde{\mu}(W^{(2)}_{q_1-n_a+i+1,q_1-n_a+i+1}t)$. 
We divide this sum into 8 piecies:
\begin{align}
&\quad\eqref{551}_1+\eqref{554}_1+\eqref{554}_6\nonumber\\
&=-\dfrac{i}{2}\hbar \widetilde{\mu}(W^{(1)}_{q_1-n_a+i,q_1-n_a+i}-W^{(1)}_{q_1-n_a+i+1,q_1-n_a+i+1}),\label{551-1}\\
&\quad\eqref{551}_2+\eqref{552}_1+\eqref{552}_2\nonumber\\
&=-\hbar\widetilde{\mu}(W^{(1)}_{q_1-n_a+i,q_1-n_a+i}W^{(1)}_{q_1-n_a+i+1,q_1-n_a+i+1}),\label{551-2}\\
&\quad\eqref{551}_3+\eqref{551}_4+\eqref{553}_1+\eqref{555}_1+\eqref{554}_3+\eqref{554}_4\nonumber\\
&=\hbar \sum_{r=1}^{b_a}\limits\displaystyle\sum_{s \geq 0}  \limits\displaystyle\sum_{u=q_1-n_a+1}^{q_1-n_a+i}\limits  e^{(r)}_{q_1-n_a+i,u}t^{-s}e^{(r)}_{u,q_1-n_a+i}t^s\nonumber\\
&\quad+\hbar\sum_{r=1}^{b_a}\limits \displaystyle\sum_{s \geq 0} \limits\displaystyle\sum_{u=q_1-n_a+i+1}^{q_1-n_{a+1}}\limits  e^{(r)}_{q_1-n_a+i,u}t^{-s-1}e^{(r)}_{u,q_1-n_a+i}t^{s+1},\label{551-3}\\
&\quad\eqref{551}_5+\eqref{551}_6+\eqref{553}_1+\eqref{555}_2+\eqref{554}_{8}+\eqref{554}_{9}\nonumber\\
&=-\hbar\sum_{r=1}^{b_a}\limits\displaystyle\sum_{s \geq 0}\limits\displaystyle\sum_{u=q_1-n_a+1}^{q_1-n_a+i}\limits e^{(r)}_{q_1-n_a+i+1,u}t^{-s} e^{(r)}_{u,q_1-n_a+i+1}t^{s}\nonumber\\
&\quad-\hbar\sum_{r=1}^{b_a}\limits\displaystyle\sum_{s \geq 0}\limits\displaystyle\sum_{u=q_1-n_a+i+1}^{q_1-n_{a+1}} \limits e^{(r)}_{q_1-n_a+i+1,u}t^{-s-1} e^{(r)}_{u,q_1-n_a+i+1}t^{s+1}\label{551-4}\\
&\quad\eqref{552}_3+\eqref{552}_5+\eqref{553}_3+\eqref{555}_5+\eqref{554}_{2}\nonumber\\
&=\hbar\sum_{1\leq r_1<r_2\leq b_a}\limits\displaystyle\sum_{s \geq 0}  \limits\displaystyle\sum_{u=q_{r_2}-n_a+1}^{q_{r_2}-n_a+i}\limits e^{(r_1)}_{q_{1}-q_{r_2}+u,q_{1}-n_a+i}t^{s}E^{(r_2)}_{q_{1}-q_{r_2}+i,q_{1}-q_{r_2}+u}t^{-s}\nonumber\\
&\quad+\hbar\sum_{1\leq r_1<r_2\leq b_a}\limits\displaystyle\sum_{s \geq 0} \limits\displaystyle\sum_{u=q_{r_2}-n_a+i+1}^{q_{r_2}-n_{a+1}}\limits e^{(r_1)}_{q_{1}-q_{r_2}+u,q_{1}-n_a+i}t^{s+1}e^{(r_2)}_{q_{1}-n_a+i,q_{1}-q_{r_2}+u}t^{-s-1},\label{551-5}\\
&\quad\eqref{552}_4+\eqref{552}_6+\eqref{553}_5+\eqref{555}_3+\eqref{554}_{5}\nonumber\\
&=\hbar\sum_{1\leq r_1<r_2\leq b_a}\limits\displaystyle\sum_{s \geq 0}  \limits\displaystyle\sum_{u=q_{r_2}-n_a+1}^{q_{r_2}-n_a+i}\limits e^{(r_1)}_{q_{1}-n_a+i,q_{1}-q_{r_2}+u}t^{-s}e^{(r_2)}_{q_{1}-q_{r_2}+u,q_{1}-n_a+i}t^s\nonumber\\
&\quad+\hbar\sum_{1\leq r_1<r_2\leq b_a}\limits\displaystyle\sum_{s \geq 0} \limits\displaystyle\sum_{u=q_{r_2}-n_a+i+1}^{q_{r_2}-n_{a+1}}\limits e^{(r_1)}_{q_{1}-n_a+i,q_{1}-q_{r_2}+u}t^{-s-1}e^{(r_2)}_{q_{1}-q_{r_2}+u,q_{1}-n_a+i}t^{s+1},\label{551-6}\\
&\quad\eqref{552}_7+\eqref{552}_9+\eqref{553}_4+\eqref{555}_6+\eqref{554}_{7}\nonumber\\
&=-\hbar\sum_{1\leq r_1<r_2\leq b_a}\limits\displaystyle\sum_{s \geq 0}\limits\displaystyle\sum_{u=q_{r_2}-n_a+1}^{q_{r_2}-n_a+i}\limits e^{(r_1)}_{q_{1}-q_{r_2}+u,q_{1}-n_a+i+1}t^{s}E^{(r_2)}_{q_{1}-n_a+i+1,q_{1}-q_{r_2}+u}t^{-s}\nonumber\\
&\quad-\hbar\sum_{1\leq r_1<r_2\leq b_a}\limits\displaystyle\sum_{s \geq 0}\limits\displaystyle\sum_{u=q_{r_2}-n_a+i+1}^{q_{r_2}-n_{a+1}} \limits e^{(r_1)}_{q_{1}-q_{r_2}+u,q_{1}-n_a+i+1}t^{s+1} e^{(r_2)}_{q_{1}-n_a+i+1,q_{1}-q_{r_2}+u}t^{-s-1},\label{551-7},\\
&\quad\eqref{552}_8+\eqref{552}_{10}+\eqref{553}_6+\eqref{555}_4+\eqref{554}_{10}\nonumber\\
&=-\hbar\sum_{1\leq r_1<r_2\leq b_a}\limits\displaystyle\sum_{s \geq 0}\limits\displaystyle\sum_{u=q_{r_2}-n_a+1}^{q_{r_2}-n_a+i}\limits e^{(r_1)}_{q_{1}-n_a+i+1,q_{1}-q_{r_2}+u}t^{-s}e^{(r_2)}_{q_{1}-q_{r_2}+u,q_{1}-n_a+i+1}t^s\nonumber\\
&\quad-\hbar\sum_{1\leq r_1<r_2\leq b_a}\limits\displaystyle\sum_{s \geq 0}\limits\displaystyle\sum_{u=q_{r_2}-n_a+i+1}^{q_{r_2}-n_{a+1}} \limits e^{(r_1)}_{q_{1}-n_a+i+1,q_{1}-q_{r_2}+u}t^{-s-1}e^{(r_2)}_{q_{1}-q_{r_2}+u,q_{1}-n_a+i+1}t^{s+1}.\label{551-8}
\end{align}
By a direct computation, we find that the suma of \eqref{551-5}-\eqref{551-8} is equal to
\begin{align*}
&\hbar \sum_{u=1}^iW^{(1)}_{q_1-n_a+i,q_{1}-n_a+u}t^{-s}W^{(1)}_{q_{1}-n_a+u,q_{1}-n_a+i}t^s\nonumber\\
&\quad+\hbar \sum_{u=i+1}^{n_a-n_{a+1}}W^{(1)}_{q_1-n_a+i,q_{1}-n_a+u}t^{-s-1}W^{(1)}_{q_{1}-n_a+u,q_{1}-n_a+i}t^{s+1}\\
&\quad-\hbar \sum_{u=1}^iW^{(1)}_{q_1-n_a+i+1,q_{1}-n_a+u}t^{-s}W^{(1)}_{q_{1}-n_a+u,q_{1}-n_a+i+1}t^s\nonumber\\
&\quad-\hbar \sum_{u=i+1}^{n_a-n_{a+1}}W^{(1)}_{q_1-n_a+i+1,q_{1}-n_a+u}t^{-s-1}W^{(1)}_{q_{1}-n_a+u,q_{1}-n_a+i+1}t^{s+1}.
\end{align*}
Thus, we have obtained the relation \eqref{goal}.
\end{proof}
\section{Affine Yangians and some cosets of non-rectangular $W$-algebras}
Restricting the homomorphism $\Phi$ to $Y_{\hbar,\ve_s}(\widehat{\mathfrak{sl}}(n_s-n_{s+1}))$, we obtain the homomorpism
\begin{equation*}
\Phi_s\colon Y_{\hbar,\ve_s}(\widehat{\mathfrak{sl}}(n_s-n_{s+1}))\to\mathcal{U}(\mathcal{W}^k(\mathfrak{gl}(N),f)).
\end{equation*}
These are homomorphisms given in Theorem 4.7 of \cite{U9}. By Theorem~\ref{Main}, we obtain the following theorem.
\begin{Theorem}\label{Com}
The images of $\Phi_s$ are commutative with each other.
\end{Theorem}
In this section, we consider these homomorphisms from the perspective of the coset of a vertex algebra. For a vertex algebra $A$ and its vertex subalgebra $B$, we set a coset vertex algebra of the pair $(A,B)$ as follows:
\begin{align*}
C(A,B)=\{v\in A\mid w_{(r)}v=0\text{ for }w\in B\text{ and }r\geq0\}.
\end{align*}
\subsection{The case that $b_1=1$}
In the case that $b_1=1$, we have an embedding from $\mathcal{U}(\widehat{\mathfrak{gl}}(n_1-n_2))$ to $\mathcal{U}(\mathcal{W}^k(\mathfrak{gl}(N),f))$ given by $E_{i,j}t^s$ to $W^{(1)}_{i,j}t^s$. Through this embedding, $\Phi_1$ becomes an evaluation map for the affine Yangian. 

We note that $\mathcal{U}(\widehat{\mathfrak{gl}}(n_1-n_2))$ can be considered as the universal enveloping algebra of the universal affine vertex algebra associated with $\mathfrak{gl}(n_1-n_2)$ and the inner product defined by
\begin{equation*}
\kappa(E_{i,j},E_{p,q})=(k+N-q_1)\delta_{i,q}\delta_{p,j}+\delta_{i,j}\delta_{p,q}.
\end{equation*}
By Theorem~\ref{Com} and Theorem~\ref{thm:main}, we obtain the following theorem by the same way as Theorem 5.9 in \cite{U10}:
\begin{Theorem}\label{1}
In the case that $k+N-q_l\neq0$, the images of $\Phi_a$ are contained in the universal enveloping algebra of $C(\mathcal{W}^k(\mathfrak{gl}(N),f),V^\kappa(\mathfrak{sl}(n_1-n_2)))$.
\end{Theorem}
\begin{Remark}
In Theorem 4.5 of \cite{U5}, we have computed the OPEs of the $W$-algebra $\mathcal{W}^k(\mathfrak{gl}(m+n),(1^{m-n},2^n))$. By using these OPEs, in the case that $l=2$ and $q_1=m>q_2=n$, we can show that the image of $\Phi_2$ is contained in the universal enveloping algebra of $C(\mathcal{W}^k(\mathfrak{gl}(m+n),f),V^\kappa(\mathfrak{sl}(m-n)))$ for any $k\in\mathbb{C}$.
We expect that Theorem~\ref{1} holds for any $k\in\mathbb{C}$.
\end{Remark}
\subsection{The case that $b_1=2$}
First, let us recall the main results of \cite{U4}, which gave $\Phi_1$ for a rectangular $W$-algebra. 
\begin{Theorem}[Theorem 3.17 and Theorem 5.1 in \cite{U4}]\label{rect}
\begin{enumerate}
\item In the case that $k+(l-1)n\neq0$, the rectangular $W$-algebra $\mathcal{W}^k(\mathfrak{gl}(ln),(l^n))$ is generated by 
$\{W^{(r)}_{i,j}\mid 1\leq i,j\leq n,r=1,2\}$.
\item The image of the homomorphism 
\begin{equation*}
\Phi_1\colon Y_{\hbar,\ve}(\widehat{\mathfrak{sl}}(n))\to\mathcal{U}(\mathcal{W}^k(\mathfrak{gl}(ln),(l^n)))
\end{equation*}
is dense in the target if $k+(l-1)n\neq0$.
\end{enumerate}
\end{Theorem}
Let us denote by $\widetilde{\mathcal{W}}^k(\mathfrak{gl}(b_1(n_1-n_2)),(b_1^{n_1-n_2}))$ the subalgebra of $\mathcal{W}^k(\mathfrak{gl}(N),f)$ generated by $\{W^{(r)}_{i,j}\mid r=1,2,1\leq i,j\leq n_1-n_2\}$.
The following theorem follows from Theorem~\ref{emb} and Theorem~\ref{rect}.
\begin{Lemma}\label{lem11}
\begin{enumerate}
\item The subalgebra $\widetilde{\mathcal{W}}^k(\mathfrak{gl}(2(n_1-n_2)),(2^{n_1-n_2}))$ is isomorphic to the rectangular $W$-algebra $\mathcal{W}^{k+n_2+\sum_{s=3}^{l}q_s}(\mathfrak{gl}(2(n_1-n_2)),(2^{n_1-n_2}))$.
\item In the case that $b_1=2$, the completion of the image of $\Phi_1$ coincides with the universal enveloping algebra of $\widetilde{\mathcal{W}}^k(\mathfrak{gl}(2(n_1-n_2)),(2^{n_1-n_2}))$.
\end{enumerate}
\end{Lemma}

By Lemma~\ref{lem11}, Theorem~\ref{Com} and Theorem~\ref{Main}, we obtain the following theorem.
\begin{Theorem}
The homomorphism $\Phi_s$ the homomorphism $\Phi_s$ induces a homomorphism
\begin{equation*}
\widetilde{\Phi}_s\colon Y_{\hbar,\ve}(\widehat{\mathfrak{sl}}(n_s-n_{s+1}))\to C(\mathcal{U}(\mathcal{W}^k(\mathfrak{gl}(N),f)),\mathcal{U}(\mathcal{W}^{k+n_2+\sum_{s=3}^{l}q_s}(\mathfrak{gl}(2(n_1-n_2)),(2^{n_1-n_2})))),
\end{equation*}
where $C(\mathcal{U}(\mathcal{W}^k(\mathfrak{gl}(N),f)),\mathcal{U}(\mathcal{W}^{k+n_2+\sum_{s=3}^{l}q_s}(\mathfrak{gl}(2(n_1-n_2)),(2^{n_1-n_2}))))$ is the centralizer algebra of $\mathcal{U}(\mathcal{W}^k(\mathfrak{gl}(N),f))$ and $\mathcal{U}(\mathcal{W}^{k+n_2+\sum_{s=3}^{l}q_s}(\mathfrak{gl}(2(n_1-n_2)),(2^{n_1-n_2})))$.
\end{Theorem}
By the same way as Theorem~6.5 in \cite{U10}, we obtain the following theorem.
\begin{Theorem}\label{rrr}
The homomorphism $\Phi_s$ the homomorphism $\Phi_s$ induces a homomorphism
\begin{equation*}
\widetilde{\Phi}_s\colon Y_{\hbar,\ve}(\widehat{\mathfrak{sl}}(n_s-n_{s+1}))\to\mathcal{U}(C(\mathcal{W}^k(\mathfrak{gl}(N),f),\mathcal{W}^{k+n_2+\sum_{s=3}^{l}q_s}(\mathfrak{sl}(2(n_1-n_2)),(2^{n_1-n_2})))).
\end{equation*}
\end{Theorem}

In the appendix, we will consider the case that $b_1>2$.
\appendix
\section{Elements of $\mathcal{W}^k(\mathfrak{gl}(N),f)$}
In this section, we construct elements of $\mathcal{W}^k(\mathfrak{gl}(N),f)$
which are correpponding to Arkawa-Molev's strong generators of rectangular $W$-algebras.

We consider two universal affine vertex algebras. The first one is associated with a Lie algebra
\begin{align*}
\mathfrak{b}&=\bigoplus_{\substack{1\leq i,j\leq N\\\col(i)\geq\col(j)}}\limits \mathbb{C}E_{i,j}\subset\mathfrak{gl}(N)
\end{align*}
and its inner product
\begin{equation*}
\kappa_{\mathfrak{b}}(E_{i,j},E_{p,q})=\alpha_{\col(i)}\delta_{i,q}\delta_{p,j}+\delta_{i,j}\delta_{p,q}.
\end{equation*}

The second one is the universal affine vertex algebra associated with a Lie superalgebra $\mathfrak{a}=\mathfrak{b}\oplus\displaystyle\bigoplus_{\substack{1\leq i,j\leq N\\\col(i)>\col(j)}}\limits\mathbb{C}\psi_{i,j}$  with the following commutator relations;
\begin{align*}
[E_{i,j},\psi_{p,q}]&=\delta_{j,p}\psi_{i,q}-\delta_{i,q}\psi_{p,j},\ 
[\psi_{i,j},\psi_{p,q}]=0,
\end{align*}
where $e_{i,j}$ is an even element and $\psi_{i,j}$ is an odd element.
We set the inner product on $\mathfrak{a}$ such that
\begin{gather*}
\widetilde{\kappa}_{\mathfrak{b}}(E_{i,j},E_{p,q})=\kappa_{\mathfrak{b}}(E_{i,j},E_{p,q}),\qquad\widetilde{\kappa}_{\mathfrak{b}}(E_{i,j},\psi_{p,q})=\widetilde{\kappa}_{\mathfrak{b}}(\psi_{i,j},\psi_{p,q})=0.
\end{gather*}
By the definition of $V^{\widetilde{\kappa}}(\mathfrak{a})$ and $V^\kappa(\mathfrak{b})$, $V^{\widetilde{\kappa}}(\mathfrak{a})$ contains $V^\kappa(\mathfrak{b})$.

In this section, we regard $V^{\widetilde{\kappa}}(\mathfrak{a})$ (resp.\ $V^\kappa(\mathfrak{b})$ and $V^{\widetilde{\kappa}}(\mathfrak{a})\otimes\mathbb{C}[\tau]$) as a non-associative superalgebra whose product $\cdot$ is defined by
\begin{equation*}
u[-w]\cdot v[-s]=(u[-w])_{(-1)}v[-s],\ [\tau,u[-s]]=su[-s],
\end{equation*}
where $\tau$ is an even element.
We sometimes omit $\cdot$ and in order to simplify the notation. By \cite{KW1} and \cite{KW2}, a $W$-algebra $\mathcal{W}^k(\mathfrak{gl}(N),f)$ can be realized as a subalgebra of $V^\kappa(\mathfrak{b})$.

Let us define an odd differential $d_0 \colon V^{\kappa}(\mathfrak{b})\to V^{\widetilde{\kappa}}(\mathfrak{a})$ determined by
\begin{gather}
d_01=0,\\
[d_0,\partial]=0,\label{ee5800}
\end{gather}
\begin{align}
[d_0,e_{i,j}[-1]]
&=\sum_{\substack{\col(i)>\col(r)\geq\col(j)}}\limits e_{r,j}[-1]\psi_{i,r}[-1]-\sum_{\substack{\col(j)<\col(r)\leq\col(i)}}\limits \psi_{r,j}[-1]e_{i,r}[-1]\nonumber\\
&\quad+\delta(\col(i)>\col(j))\alpha_{\col(i)}\psi_{i,j}[-2]+\psi_{\hat{i},j}[-1]-\psi_{i,\tilde{j}}[-1].\label{ee1}
\end{align}
By using Theorem 2.4 in \cite{KRW}, we can define the $W$-algebra $\mathcal{W}^k(\mathfrak{gl}(N),f)$ as follows.
\begin{Definition}\label{T125}
The $W$-algebra $\mathcal{W}^k(\mathfrak{gl}(N),f)$ is the vertex subalgebra of $V^\kappa(\mathfrak{b})$ defined by
\begin{equation*}
\mathcal{W}^k(\mathfrak{gl}(N),f)=\{y\in V^\kappa(\mathfrak{b})\subset V^{\widetilde{\kappa}}(\mathfrak{a})\mid d_0(y)=0\}.
\end{equation*}
\end{Definition}
\begin{Remark}
The projection from $\mathfrak{b}$ to $\bigotimes_{1\leq i\leq l}\mathfrak{gl}(q_i)$ induces the Miura map $\mu$, that is, an embedding from $\mathcal{W}^k(\mathfrak{gl}(N),f)$ to $\bigotimes_{1\leq i\leq l}V^{\kappa_i}(\mathfrak{gl}(q_i))$. 
\end{Remark}
By using Definition~\ref{T125}, we construct some elements of $\mathcal{W}^k(\mathfrak{gl}(N),f)$. 

We denote by $T(C)$ a non-associative free algebra associated with a vector space $C$ and by $\mathfrak{gl}(l)_{\leq0}$ the Lie algebra $\displaystyle\bigoplus_{\substack{1\leq j\leq i\leq l}}\limits\mathbb{C}E_{i,j}$.
Let us set an $q_1\times q_1$ matrix $B=(b_{i,j})_{1\leq i,j\leq q_1}$ as
\begin{equation}
\begin{bmatrix}
\alpha_1\tau+E_{1,1}[-1] &-1\phantom{-}&0&\dots & 0\\[0.4em]
E_{2,1}[-1] &\alpha_1\tss\tau+E_{2,2}[-1] &-1\phantom{-}&\dots & 0\\[0.4em]
\vdots &\vdots &\ddots & &\vdots\\[0.4em]
E_{l-1,1}[-1] &E_{l-1,2}[-1] &\dots
&\alpha_1\tss\tau+E_{l-1,l-1}[-1] &-1\phantom{-}\\[0.4em]
E_{l, 1}[-1] &E_{l, 2}[-1] &\dots &E_{l,l-1}[-1]  &\pi\tss\tau+E_{l,l}[-1]
\end{bmatrix}\label{matrx}
\end{equation}
whose entries are elements of $T(\mathfrak{gl}(l)_{\leq0}[t^{-1}]t^{-1})\otimes \mathbb{C}[\tau]$. 
For any matrix $A=(a_{i,j})_{1\leq i,j\leq s}$, we define $\text{cdet}(A)$ as 
\begin{equation*}
\displaystyle\sum_{\sigma\in\mathfrak{S}_s}\limits\text{sgn}(\sigma)a_{\sigma(1),1}\big(a_{\sigma(2),2}(a_{\sigma(3),3}\cdots a_{\sigma(s-1),s-1})a_{\sigma(s),s}\big)\in T(\mathfrak{gl}(l)_{\leq0}[t^{-1}]t^{-1})\otimes\mathbb{C}[\tau].
\end{equation*}
We regard $\mathfrak{gl}(q_1)$ as an associative superalgebra whose product $\cdot$ is determined by $E_{i,j}\cdot E_{s,u}=\delta_{j,s}E_{i,u}$.
Then, we obtain a non-associative superalgebra $\mathfrak{gl}(q_1)\otimes V^\kappa(\mathfrak{b})\otimes\mathbb{C}[\tau]$.
We construct a homomorphism 
\begin{equation*}
T\colon T(\mathfrak{gl}(l)_{\leq0}[t^{-1}]t^{-1})\otimes\mathbb{C}[\tau]\to \mathfrak{gl}(m|n)\otimes V^\kappa(\mathfrak{b})\otimes\mathbb{C}[\tau]
\end{equation*}
determined by
\begin{gather*}
T_{i,j}(E_{u,v}[-s])=E_{(u-1)q_1+i,(v-1)q_1+j}[-s]\in\mathfrak{b}[t^{-1}]t^{-1},\quad T(\tau)=\tau,
\end{gather*}
where $T_{i,j}(x)$ is defined as $E_{j,i}\otimes T_{i,j}(x)=T(x)$. Since $T$ is a homomorphism, we obtain
\begin{gather*}
T_{i,j}(xy)=\sum_{r=1}^{q_1}T_{r,i}(x)T_{j,r}(y).
\end{gather*}
By the commutator relation of $V^\kappa(\mathfrak{b})$ and $\mathbb{C}[\tau]$, we define $\widetilde{W}^{(r)}_{i,j}\in V^{\kappa_{\mathfrak{b}}}(\mathfrak{b})$ is defined by
\begin{equation}
T_{j,i}(\text{cdet}(B))=\sum_{r=0}^l\limits \widetilde{W}^{(r)}_{j,i}(\alpha_1\tau)^{l-r}.\label{5261}
\end{equation}
\begin{Theorem}\label{gen}
The $W$-algebra $\mathcal{W}^k(\mathfrak{gl}(N),f)$ contains $\{\widetilde{W}^{(r)}_{i,j}\mid1\leq r\leq l,1\leq i,j\leq n_1-n_2\}$.
\end{Theorem}
The proof of Theorem~\ref{gen} is as the one of Arakawa-Molev \cite{AM} (see also Theorem~3.11 and Remark~3.9) since we take $1\leq i,j\leq n_1-n_2$.

In particular, by \eqref{5261}, we have
\begin{align}
\mu(\widetilde{W}^{(1)}_{i,j})&=\displaystyle\sum_{1\leq s\leq b_1}\limits e_{(s-1)q_1+i,(s-1)q_1+j}[-1],\label{W1}\\
\mu(\widetilde{W}^{(2)}_{i,j})&=\alpha_1\displaystyle\sum_{1\leq s\leq b_1}\limits (s-1)e_{(s-1)q_1+i,(s-1)q_1+j}[-2]\nonumber\\
&\quad+\displaystyle\sum_{\substack{r_1<r_2\\1\leq t\leq m+n}}\limits{(-1)}^{p(t)+p(e_{i,t})p(e_{j,t})} e^{(r_1)}_{t,i}[-1]e^{(r_2)}_{j,t}[-1].\label{W2}
\end{align}

By the form of $W^{(1)}_{i,j}$ and $W^{(2)}_{i,j}$, we find that the OPEs of $\{W^{(r)}_{i,j}\mid r=1,2,1\leq i,j\leq n_1-n_2\}$ in $\mathcal{W}^k(\mathfrak{gl}(N),f)$ is the same as those of
\begin{equation}
\{\widetilde{W}^{(1)}_{i,j},W^{(2)}_{i,j}-\gamma_{1}\partial W^{(1)}_{i,j}\mid r=1,2,1\leq i,j\leq n_1-n_2\}\label{OPE}
\end{equation}
in $\mathcal{W}^{k+\sum_{s=3}^{l}q_s}(\mathfrak{gl}(q_1b_1),(b_1^{q_1}))$. Similarly to Theorem~\ref{rrr}, we find the following theorem.
\begin{Theorem}
The homomorphism $\Phi_s$ the homomorphism $\Phi_s$ induces a homomorphism
\begin{equation*}
\Phi\colon Y_{\hbar,\ve}(\widehat{\mathfrak{sl}}(n_s-n_{s+1}))\to C(\mathcal{U}(\mathcal{W}^k(\mathfrak{gl}(N),f)),\mathcal{U}(\widetilde{\mathcal{W}}^k(\mathfrak{gl}(b_1(n_1-n_2)),(b_1^{n_1-n_2}))),
\end{equation*}
where $C(\mathcal{U}(\mathcal{W}^k(\mathfrak{gl}(N),f)),\mathcal{U}(\widetilde{\mathcal{W}}^k(\mathfrak{gl}(b_1(n_1-n_2)),(b_1^{n_1-n_2})))$ is the centralizer of $\mathcal{U}(\mathcal{W}^k(\mathfrak{gl}(N),f))$ and $\mathcal{U}(\widetilde{\mathcal{W}}^k(\mathfrak{gl}(b_1(n_1-n_2)),(b_1^{n_1-n_2})))$.
\end{Theorem}
The OPEs of the elements \eqref{OPE} in the rectangular $W$-algebra $\mathcal{W}^{k+\sum_{s=b_1+1}^{l}q_s}(\mathfrak{gl}(q_1b_1),(b_1^{q_1}))$ is the same as those in $\mathcal{W}^{k+n_2(b_1-1)+\sum_{s=b_1+1}^{l}q_s}(\mathfrak{gl}((n_1-n_2)b_1),(b_1^{n_1-n_2}))$. 
\begin{Conjecture}
The subalgebra $\widetilde{\mathcal{W}}^k(\mathfrak{gl}(b_1(n_1-n_2)),(b_1^{n_1-n_2}))$ is isomorphic to the rectangular $W$-algebra $\mathcal{W}^{k+n_2(b_1-1)+\sum_{s=b_1+1}^{l}q_s}(\mathfrak{gl}((n_1-n_2)b_1),(b_1^{n_1-n_2}))$.
\end{Conjecture}
If we can prove this conjecture, we can obtain the homomrphism
\begin{equation*}
\widetilde{\Phi}_s\colon Y_{\hbar,\ve}(\widehat{\mathfrak{sl}}(n_s-n_{s+1}))\to\mathcal{U}(C(\mathcal{W}^k(\mathfrak{gl}(N),f),\mathcal{W}^{k+(b_1-1)n_2+\sum_{s=3}^{l}q_s}(\mathfrak{sl}(b_1(n_1-n_2)),(b_1^{n_1-n_2})))
\end{equation*}
by the same way as Theorem~6.5 in \cite{U10}.
\bibliographystyle{plain}
\bibliography{syuu}
\end{document}